\providecommand{\StandardLayout}{}
\documentclass[mmnp]{edpsmath}
\usepackage{stmaryrd} % to use \llbracket and \rrbracket
\usepackage{graphicx} % for the figures

\newtheorem{theorem}{Theorem}[section]
\newtheorem{definition}[theorem]{Definition}
\newtheorem{lemma}[theorem]{Lemma}
\newtheorem{proposition}[theorem]{Proposition}
\newtheorem{remark}[theorem]{Remark}
\newtheorem{assumption}[theorem]{Assumption}

\newcommand{\N}{\mathbb{N}}
\newcommand{\R}{\mathbb{R}}
\newcommand{\C}{\mathbb{C}}
\newcommand{\X}{\mathcal{X}}
\newcommand{\Co}{\mathcal{C}}
\newcommand{\OO}{\mathcal{O}}
\newcommand{\A}{\mathcal{A}}
\renewcommand{\S}{\mathcal{S}}
\newcommand{\supp}{\text{supp}}
\newcommand{\ep}{\varepsilon}
\newcommand{\M}{\mathcal{M}}
\newcommand{\Ker}{\textnormal{Ker}}
\newcommand{\Rg}{\mathcal{R}}

\begin{document}
%%-----------------------------
%%      the top matter
%%-----------------------------
\title{Global stability in a competitive infection-age structured model}

\runningtitle{Global stability in a competitive infection-age structured model}

\author{Quentin Richard}\address{
Université de Bordeaux, IMB, UMR CNRS 5251, F-33400 Talence, France. Email: quentin.richard@math.cnrs.fr}

\runningauthors{Quentin Richard}

%
%\date{}
%
\begin{abstract} We study a competitive infection-age structured SI model between two diseases. The well-posedness of the system is  handled by using integrated semigroups theory, while the existence and the stability of disease-free or endemic equilibria are ensured,  depending on the basic reproduction number $R_0^x$ and $R_0^y$ of each strain. We then exhibit Lyapunov functionals to analyse the global stability and we prove that the disease-free equilibrium is globally asymptotically stable whenever $\max\{R_0^x, R_0^y\}\leq 1$. With respect to explicit basin of attraction, the competitive exclusion principle occurs in the case where $R_0^x\neq R_0^y$ and $\max\{R_0^x,R_0^y\}>1$, meaning that the strain with the largest $R_0$ persists and eliminates the other strain. In the limit case $R_0^x=R^0_y>1$, an infinite number of endemic equilibria exists and constitute a globally attractive set. \end{abstract}
%
%\begin{resume} \end{resume}
%
\subjclass{35B35, 35B40, 47D62, 92D30}
\keywords{Lyapunov function, integrated semigroup, global stability, dynamical systems, structured population dynamics, competitive exclusion}
\maketitle

\section{Introduction}\label{Sec:Intro}

In \cite{Kermack27}, Kermack and McKendrick proposed the first ODE epidemic model. Since then, the literature on this topic is wide and such models are commonly used to predict the evolution of a disease and eventually prevent the apparition of epidemics. 
Incorporating another continuous variable such as the age since infection \cite{MagClusWebb2010, MagClusk2013, MartchevaLi2013,ThiemeCC93}, the infection-load \cite{Perasso2019, PerassoRazafison2014} or the time remaining before disease detection \cite{LarochePerasso2016}, the so-called structured epidemiological models are described by transport equations (we refer \textit{e.g.} to \cite{Brauer2012, Iannelli2017} for an introduction of such models) and sometimes by transport-diffusion equations \cite{CalsinaFarkas2012, CalsinaFarkas2016}. In the present paper, we consider the following infection-age structured SI model, that describes the competition between two diseases for a same susceptible population:
\begin{equation}
\label{Eq:Model}
\left\{
\begin{array}{rcl}
\dfrac{dS}{dt}(t)&=&\Lambda-\mu_S S(t)-S(t) \int_0^\infty\beta_x(a)x(t,a)da-S(t)\int_0^\infty \beta_y(a)y(t,a)da, \\
\dfrac{\partial x}{\partial t}(t,a)+\dfrac{\partial x}{\partial a}(t,a)&=&-\mu_x(a)x(t,a), \\
x(t,0)&=&S(t) \int_0^\infty \beta_x(a)x(t,a)da \\
\dfrac{\partial y}{\partial t}(t,a)+\dfrac{\partial y}{\partial a}(t,a)&=&-\mu_y(a)y(t,a), \\
y(t,0)&=&S(t) \int_0^\infty \beta_y(a)y(t,a)da \\
(S(0), x(0,\cdot), y(0,\cdot))&=&(S_0, x_0, y_0)\in \R_+\times L^1_+(0,\infty)\times L^1_+(0,\infty)
\end{array}
\right.
\end{equation}
for every $t\geq 0$ and $a\geq 0$. Such system can for example be used to describe competition between two strains of a same disease, as influenza \cite{DangMartcheva2016}, malaria \cite{DjidjouDucrot2013} or avian influenza \cite{MartchevaLi2013}. It can also be used in other contexts as competition between species for a same nutrient in a chemostat \cite{SmithThieme2013}, or competition between predators for a single ressource \cite{DucrotMadec2013}.

Here, $S(t), x(t,a)$ and $y(t,a)$ respectively denote the density of susceptible individuals at time $t$ and both infected populations of age $a$ and at time $t$. The parameter $\Lambda$ represents the recruitment flux into the susceptible class while $\mu_S, \mu_x$ and $\mu_y$ are the mortality rates of the three populations. Finally $\beta_x$ and $\beta_y$ describe the transmission rates of both infected populations $x$ and $y$. Let
$$\overline{\beta_x}=\sup(\supp(\beta_x)), \qquad \overline{\beta_y}=\sup(\supp(\beta_y))$$ ($\supp(\cdot)$ denoting the support of any function) be the maximal age of infectiousness of the corresponding disease. In the sequel, we will make the following assumption.

\begin{assumption}\label{Assump1}\mbox{}
\begin{enumerate}
\item The parameters $\Lambda, \mu_S>0$ are positive and the functions $\mu_x$, $\mu_y$, $\beta_x, \beta_y$ are in $L^\infty(0,\infty)$ with $\beta_x\not\equiv 0$ and $\beta_y\not\equiv 0$. Moreover there exists $\mu_0>0$ such that:
$$\min\{\mu_S, \mu_x(a), \mu_y(a)\}\geq \mu_0 \ \text{a.e. } a\geq 0.$$
\item There exist $\underline{\beta_x}\in [0,\overline{\beta_x})$ and $\underline{\beta_y}\in [0,\overline{\beta_y})$ such that
$$\beta_x(a)>0 \quad \text{a.e.} \quad a\in [\underline{\beta_x}, \overline{\beta_x}), \qquad \beta_y(a)>0 \quad   \text{a.e} \quad  a\in [\underline{\beta_y}, \overline{\beta_y}).$$
\end{enumerate}
\end{assumption}
Consequently to the latter assumption, the probabilities functions 
$$\pi_{x}:a\mapsto e^{-\int_0^a \mu_x(s)ds}, \quad \pi_{y}:a\mapsto e^{-\int_0^a \mu_y(s)ds}$$
describe the survival of the corresponding infected population. In the case $\beta_y\equiv 0$, the system \eqref{Eq:Model} becomes the following infection-age structured model with only one disease:
\begin{equation}
\label{Eq:Model_1pop}
\left\{
\begin{array}{rcl}
\dfrac{dS(t)}{dt}&=&\Lambda-\mu_S S(t)-S(t) \int_0^\infty\beta_x(a)x(t,a)da, \\
\dfrac{\partial x(t,a)}{\partial t}+\dfrac{\partial x(t,a)}{\partial a}&=&-\mu_x(a)x(t,a), \\
x(t,0)&=&S(t) \int_0^\infty \beta_x(a)x(t,a)da.
\end{array}
\right.
\end{equation}
This latter model \eqref{Eq:Model_1pop} has been investigated by Thieme and Castillo-Chavez \cite{ThiemeCC89, ThiemeCC93} with the study of the uniform persistence and local exponential asymptotic stability of the endemic equilibrium. Related epidemic models with delay can be found \textit{e.g.} in \cite{McCluskey2009, McCluskey2010}. Thereafter, Magal, McCluskey and Webb \cite{MagClusWebb2010} handled the global stability of the endemic equilibrium of \eqref{Eq:Model_1pop}, by proving the result below. First define the quantity
$$R_0=\dfrac{\Lambda \int_0^\infty \beta_x(a)e^{-\int_0^a \mu_x(s)ds}da}{\mu_S}$$
describing the number of secondary infections produced by a single infected patient. This latter threshold is commonly used in the litterature (see \textit{e.g.} \cite{Heesterbeek2002} or more recently \cite{Perasso2018} for an introduction). First appareared in a demographic context with the work of Dublin and Lotka \cite{DublinLotka25} (see more recently \cite[Chapter 9]{Inaba2017} for more references), it is now frequently used in epidemiology (see \textit{e.g.} \cite{Diekmann90, Diekmann00}) to state if a disease will persist or disappear. 
\begin{proposition} \label{Prop:Magal}
Suppose that Assumption \ref{Assump1} holds. If $R_0\leq 1$, then \eqref{Eq:Model_1pop} admits only the disease-free equilibrium $(\frac{\Lambda}{\mu_S},0)\in \R_+\times L^1_+(0,\infty)$, while if $R_0>1$ then there exists also a (unique) endemic equilibrium denoted by $E^*$. Moreover, if $R_0<1$ (resp. $R_0=1$), then $(\frac{\Lambda}{\mu_S},0)$ is globally asymptotically stable in $\R_+\times L^1_+(0,\infty)$ (resp. globally attractive). If $R_0 >1$ then the equilibrium $E^*$ is globally asymptotically stable in the set
$$\S:=\left\{(S_0,x_0)\in \R_+\times L^1_+(0,\infty): \int_0^{\overline{\beta_x}}x_0(s)ds>0\right\}$$
while the disease-free equilibrium $(\frac{\Lambda}{\mu_S},0)$ is globally attractive in $(\R_+\times L^1_+(0,\infty))\setminus \S$.
\end{proposition}
The same result holds when interchanging the indexes $x$ and $y$. At this point we can note that in \cite{MagClusWebb2010}, the authors mentioned the global asymptotic stability of the disease-free equilibrium in the delicate case $R_0=1$. However, it seems that only the attractiveness is proved, by using Lyapunov functional. The same lack of proof seems to appear also \textit{e.g.} in \cite{MartchevaLi2013, Perasso2019}. The reason for this is twofold. Firstly, in infinite dimensional systems, the stability property is not ensured even if the attractiveness property is (see the Lasalle invariance principle \cite{MiSmiThi95}). Secondly, the principle of linearisation used to get the local asymptotic stability fails when $R_0=1$: indeed, we obtain eigenvalues with real part equals to zero. However, we will show in Section \ref{Sec:Lyap_Stab} how to overcome the stability in that case, by using some Lyapunov functional.

Recently, some papers considered structured epidemiological models with two groups of infections, or two paths of infection (see \textit{e.g.} \cite{ChengDongTak2018, DaiZou2015, MagClusk2013}) by adding some interaction between the two groups. The global stability of the equilibria and the persistence of the diseases are investigated, leading the to the existence of a $R_0$ threshold.

A very similar model to \eqref{Eq:Model} was analysed by Martcheva and Li \cite{MartchevaLi2013}, where they considered a SIR model with $n\geq 2$ different groups of infectious individuals, to see how the emergence of other diseases can affect the dynamics of the susceptible population. The analyse leads to the existence of $n$ thresholds, one for each disease. Then, using persistence results and proving existence of a global attractor as in \cite{MagClusWebb2010}, they enlighten a competitive exclusion principle, meaning that the disease with the biggest $R_0$ value will asymptotically survive, while the other strains will disappear. This fundamental result in ecology was first postulated by Gause \cite{Gause34}. We refer \textit{e.g.} to \cite{CaiMartcheva2013, DangMartcheva2016, DjidjouDucrot2013} for similar structured models where this principle occurs.

We first define the following thresholds 
$$R^x_0 :=\dfrac{\Lambda r_x}{\mu_S}, \qquad R^y_0 :=\dfrac{\Lambda r_y}{\mu_S}$$
where 
$$r_x:=\int_0^\infty \beta_x(a)\pi_{x}(a)da>0, \qquad r_y:=\int_0^\infty \beta_y(a)\pi_{y}(a)da>0.$$
The system \eqref{Eq:Model} always admits the disease free equilibrium
$$E_0:=\left(S^*_0, 0, 0\right)=\left(\dfrac{\Lambda}{\mu_S},0,0\right).$$
When $R_0^x>1$ (resp. $R_0^y>1$), we also have an endemic equilibrium given by
$$E_1:=(S^*_1, x^*_1, 0), \qquad (\textnormal{resp. } \ E_2:=(S^*_2, 0, y^*_2))$$
where
\begin{equation*}
\left\{
\begin{array}{rcl}
S^*_1&=&\dfrac{1}{r_x} \vspace{0.1cm} \\
x^*_1(a)&=& \dfrac{\mu_S(R^x_0 -1)}{r_x}\pi_{x}(a),
\end{array}
\right.
\begin{array}{rcl}
S^*_2&=&\dfrac{1}{r_y} \vspace{0.1cm}\\
y^*_2(a)&=& \dfrac{\mu_S(R^y_0 -1)}{r_y}\pi_{y}(a),
\end{array}
\end{equation*}
for every $a\geq 0$. Finally, when $R_0^x=R_0^y>1$, we have an infinite number of equilibria, given by
$$E^*_\alpha=(S^*, x^*_{\alpha}, y^*_{\alpha}), \quad \forall \alpha\in[1,2]$$
with
\begin{equation*}
\left\{
\begin{array}{rcl}
S^*&=&\dfrac{1}{r_x}=\dfrac{1}{r_y} \vspace{0.1cm}\\
x^*_{\alpha}(a)&=&\dfrac{\mu_S(R^x_0 -1)}{r_x}(2-\alpha) \pi_x(a) \vspace{0.1cm} \\
y^*_{\alpha}(a)&=& \dfrac{\mu_S(R^y_0 -1)}{r_y}(\alpha-1)\pi_{y}(a)
\end{array}
\right.
\end{equation*}
for every $a\geq 0$ and where we can note that $E^*_1=E_1$ and $E^*_2=E_2$.
In order to analyse the asymptotic behaviour of the solutions, we let $\X_+=\R_+\times L^1_+(0,\infty)\times L^1_+(0, \infty)$ and we define the sets
$$\S_x:=\{(S_0,x_0,y_0)\in \X_+: \int_0^{\overline{\beta_x}} x_0(s)ds >0\}, \qquad \partial \S_x=\X_+\setminus \S_x,$$
$$\S_y:=\{(S_0,x_0,y_0)\in \X_+: \int_0^{\overline{\beta_y}} y_0(s)ds >0\}, \qquad \partial \S_y=\X_+\setminus \S_y$$
containing initial infected populations that are in age to contaminate susceptible individuals, with the corresponding disease, now or in the future. The convergence results, obtained in the present paper, that depend on the thresholds $R_0^x, R_0^y$ and on the initial condition, are summed up in the following table: 
\begin{figure}[!hbtp]\label{Fig:Table}
\begin{center}
\begin{tabular}{|c|c|c|c|c|}
\hline 
 & $\partial \S_x\cap \partial \S_y$ & $\S_x\cap \partial \S_y$ & $\partial \S_x\cap \S_y$  & $\S_x\cap \S_y$ \\ 
\hline 
$\max\{R_0^x, R_0^y\}\leq 1$ & $E_0$ & $E_0$ & $E_0$ & $E_0$ \\ 
\hline 
$R_0^x>1\geq R_0^y$ & $E_0$ & $E_1$ & $E_0$ & $E_1$ \\ 
\hline 
$R_0^y>1\geq R_0^x$  & $E_0$ & $E_0$ & $E_2$ & $E_2$ \\ 
\hline 
$R_0^x>R_0^y>1$  & $E_0$ & $E_1$ & $E_2$ & $E_1$ \\ 
\hline 
$R_0^y>R_0^x>1$ & $E_0$ & $E_1$ & $E_2$ & $E_2$ \\ 
\hline 
$R_0^x=R_0^y>1$ & $E_0$ & $E_1$ & $E_2$ & $\{E^*_\alpha, \alpha \in[1,2]\}$ \\ 
\hline 
\end{tabular} 
\caption{Convergence of the solutions depending on $R_0^x, R_0^y$ and on the initial condition.}
\end{center}
\end{figure}

We notice that for each $k\in\{x,y\}$, when taking an initial condition in $\partial \S_k$  the solutions behave as in the case \eqref{Eq:Model_1pop}, that is to say either the initial condition is taken in $\partial \S_k$ and the solution goes to $E_0$, or it is taken in $\S_k$ and the fate of the solution depends on the threshold $R_0^k$.  Furthermore, we prove that for each value $R_0^x$ and $R_0^y$, the equilibria $E_0$, $E_1$ and $E_2$ are globally asymptotically stable in the corresponding basin of attraction, according to Figure \ref{Fig:Table}. The competitive exclusion principle is then verified and we also  handle \textit{e.g.} the global asymptotic stability of $E_0$ in $\X_+$ when $\max\{R_0^x, R_0^y\}=1$. However, the stability of the set of equilibria $\{E^*_\alpha, \alpha \in[1,2]\}$ is left open.

We first use the integrated semigroup theory, following \cite{MagClusWebb2010}, to get an appropriate framework in order to prove that \eqref{Eq:Model} is well-posed. It also allows us to linearise the system around each equilibrium, obtaining linear $C_0$-semigroups, then we use spectral theory to get the local stability of the equilibria (see \textit{e.g.} \cite{EngelNagel2000, Webb85, Webb87} for more results on this topic). In \cite{MagClusWebb2010}, the authors combine uniform persistence results due to Hale and Waltman \cite{HaleWaltman89}, with results obtained in \cite{MagalZhao2005}, to get the existence of a global attractor. While the same approach was used in \cite{MartchevaLi2013}, we follow \cite{Perasso2019} and we take advantages of an explicit formulation of the semiflow that enables us to exhibit the compactness of the orbits.

The method then used to perform the global analysis is based on the existence of a Lyapunov function (see \textit{e.g.} \cite{Hsu2005} for a survey of such functions in various ecological ODE and reaction-diffusion models). We therefore use the following key non-negative function:

\begin{equation}\label{Eq:Fct_g}
\begin{array}{rcl}
g:\R^*_+&\longrightarrow& \R \\
x&\longmapsto& x-\ln(x)-1
\end{array}
\end{equation}
that was first used by Goh \cite{Goh77} and Hsu \cite{Hsu78}. For the present model, we shall also use the following Volterra-type Lyapunov, incorporating the age-structure:

\begin{equation*}
\phi\longmapsto x^*\int_0^\infty \Psi(a)x^*(a)g\left(\dfrac{\phi(a)}{x^*(a)}\right)da
\end{equation*}
for any function $\phi>0$ a.e. with $x^*$ the equilibrium and $\phi$ some appropriate function. It was introduced in \cite{MagClusWebb2010}, and was later used \textit{e.g.} in  \cite{DjidjouDucrot2013, MagClusk2013, MartchevaLi2013, Perasso2019} for structured models. Note that similar functionals are used for delayed equations (see \textit{e.g.} \cite{PerassoRichard19b} and the references therein). The latter attractiveness combined with the stability then yield the global asymptotic stability of the corresponding equilibrium.

Note that the technique used in the present paper, contrarily to \cite{MartchevaLi2013}, allows us to study the case where the maximal reproduction number is not unique, that is when $R_0^x=R_0^y$. As written in Figure \ref{Fig:Table}, the set of equilibria $\{E^*_\alpha, \alpha\in[1,2]\}$ is proved to be globally attractive in $\S_x\cap \S_y$. Finally, following \cite{Gabriel2012} and \cite{PerassoRichard19b}, we handle the stability of the disease-free equilibrium $E_0$ in the case $\max\{R_0^x, R_0^y\}=1$, by making use of the Lyapunov functionals.

This article is structured as follows: in Section \ref{Sec:Well-Posed} we give the preliminaries results concerning existence, uniqueness and boundedness of the solutions. In Section \ref{Sec:Eq} we handle the stability of each equilibrium. Section \ref{Sec:Attractor} then deals with the existence of a compact attractor for the dynamical system and the identification of the basins of attraction. In Section \ref{Sec:Global} we investigate the global analysis of \eqref{Eq:Model}. We start by defining suitable Lyapunov functionals and proving their well-posedness. It allows to prove on one hand the global attractiveness of each equilibrium, by using a Lasalle invariance principle theorem, and on the other hand the stability of the disease free equilibrium when the principle of linearisation fails. Finally, we conclude about the global stability of each equilibrium. We end the paper with some numerical simulations in  Section \ref{Sec:Simu} to illustrate the above results.

\section{Well-posedness} \label{Sec:Well-Posed}

\subsection{Integrated semigroup formulation}

In this section, we handle the well-posedness of \eqref{Eq:Model}. To this end, we follow \cite{MagClusWebb2010} and we use integrated semigroups theory (see \textit{e.g.} \cite{MagalRuan2018} and the references therein for more details), whose approach was introduced by Thieme \cite{Thieme90}. First we consider the space
$$\hat{X}=\R\times L^1(0,\infty)$$
then we define the linear operators $\hat{A}_x:D(\hat{A}_k)\subset \hat{X}\to \hat{X}$ and $\hat{A}_y:D(\hat{A}_k)\subset \hat{X}\to \hat{X}$ by
$$\hat{A}_x\begin{pmatrix}
0 \\
\phi
\end{pmatrix}=\begin{pmatrix}
-\phi(0) \\
-\phi'-\mu_x \phi
\end{pmatrix}, \qquad \hat{A}_y\begin{pmatrix}
0 \\
\phi
\end{pmatrix}=\begin{pmatrix}
-\phi(0) \\
-\phi'-\mu_y \phi
\end{pmatrix}$$
with 
$$D\left(\hat{A}_x\right)=D\left(\hat{A}_y\right)=\{0\}\times W^{1,1}(0,\infty).$$
If $\lambda\in \C$ is such that $\Re(\lambda)>-\mu_0$, then $\lambda\in \rho(\hat{A}_x)\cap \rho(\hat{A}_y)$ (the resolvent sets of $\hat{A}_x$ and $\hat{A}_y$ respectively), and we have the following explicit formula for the resolvent of $\hat{A}_k$ (with $k\in\{x,y\}$):
\begin{equation}\label{Eq:Resolv}
\left(\lambda I-\hat{A}_k\right)^{-1}\begin{pmatrix}
c\\
\psi
\end{pmatrix}=\begin{pmatrix}
0 \\
\phi
\end{pmatrix}\Longleftrightarrow \phi(a)=c e^{-\int_0^a(\mu_k(s)+\lambda)ds}+\int_0^a e^{-\int_s^a (\mu_k(\xi)+\lambda)d\xi}\psi(s)ds.
\end{equation}
We can notice that \eqref{Eq:Model} is equivalent to
\begin{equation}\label{Eq:Model_0}
\left\{
\begin{array}{rcl}
S'(t)&=&\Lambda-\mu_S S(t)-S(t)\int_0^\infty \beta_x(a)x(t,a)da-S(t)\int_0^\infty \beta_y(a)y(t,a)da, \\
\dfrac{d}{dt}\begin{pmatrix}
0 \\
x(t,\cdot)
\end{pmatrix}&=&\hat{A}_x\begin{pmatrix}
0 \\
x(t,\cdot)
\end{pmatrix}+\begin{pmatrix}
S(t)\int_0^\infty \beta_x(a)x(t,a)da \\
0
\end{pmatrix}, \\
\dfrac{d}{dt}\begin{pmatrix}
0 \\
y(t,\cdot)
\end{pmatrix}&=&\hat{A}_y\begin{pmatrix}
0 \\
y(t,\cdot)
\end{pmatrix}+\begin{pmatrix}
S(t)\int_0^\infty \beta_y(a)y(t,a)da \\
0
\end{pmatrix}, \\
(S(0),x(0,\cdot), y(0,\cdot))&=&(S_0, x_0, y_0)\in \R_+\times L^1_+(0,\infty)\times L^1_+(0,\infty).
\end{array}
\right.
\end{equation}
Defining
$$\hat{x}(t)=\begin{pmatrix}
0 \\
x(t,\cdot)
\end{pmatrix}, \quad \hat{y}(t)=\begin{pmatrix}
0 \\
y(t,\cdot)
\end{pmatrix}$$
we can then rewrite \eqref{Eq:Model_0} as an ordinary differential equation coupled with two non-densely defined Cauchy problem:
\begin{equation*}
\left\{
\begin{array}{rcl}
\dfrac{dS}{dt}&=&-\mu_S S(t)+F_1(S(t),\hat{x}(t), \hat{y}(t)), \vspace{0.1cm} \\
\dfrac{d\hat{x}(t)}{dt}&=&\hat{A}_x \hat{x}(t)+F_2(S(t), \hat{x}(t), \hat{y}(t)),  \vspace{0.1 cm} \\
\dfrac{d\hat{y}(t)}{dt}&=&\hat{A}_y \hat{y}(t)+F_3(S(t), \hat{x}(t), \hat{y}(t)),
\end{array}
\right.
\end{equation*}
where
$$F_1\left(S, \begin{pmatrix}
0 \\
x
\end{pmatrix}, \begin{pmatrix}
0 \\
y
\end{pmatrix}\right)=\Lambda-S\int_0^\infty \beta_x(a)x(a)da-S\int_0^\infty \beta_y(a)y(a)da,$$
$$F_2\left(S, \begin{pmatrix}
0 \\
x
\end{pmatrix}, \begin{pmatrix}
0 \\
y
\end{pmatrix}\right)=\begin{pmatrix}
S\int_0^\infty \beta_x(a)x(a)da \\
0
\end{pmatrix}$$
and
$$F_3\left(S, \begin{pmatrix}
0 \\
x
\end{pmatrix}, \begin{pmatrix}
0 \\
y
\end{pmatrix}\right)=\begin{pmatrix}
S\int_0^\infty \beta_y(a)y(a)da \\
0
\end{pmatrix}.$$
Consider the sets
$$X=\R\times \left(\R\times L^1(0,\infty)\right)^2, \qquad X_+=\R_+\times \left(\R_+\times L^1_+(0,\infty)\right)^2$$
and define the linear operator $A:D(A)\subset X\to X$ by
$$A\begin{pmatrix}
S \\
\begin{pmatrix}
0 \\
x
\end{pmatrix} \\
\begin{pmatrix}
0 \\
y
\end{pmatrix}
\end{pmatrix}=\begin{pmatrix}
-\mu_S S \\
\hat{A}_x \begin{pmatrix}
0 \\
x
\end{pmatrix} \\
\hat{A}_y \begin{pmatrix}
0 \\
y
\end{pmatrix}
\end{pmatrix}
$$
with
$$D(A)=\R\times D(\hat{A}_x)\times D(\hat{A}_y).$$
We then see that
$$\overline{D(A)}=\R\times \left(\{0\}\times L^1(0,\infty)\right)^2$$
(the closure of $D(A)$), so that $D(A)$ is not dense in $X$. Now, define the non-linear function $F:\overline{D(A)}\to X$ by
$$F\begin{pmatrix}
S \\
\begin{pmatrix}
0 \\
x
\end{pmatrix} \\
\begin{pmatrix}
0 \\
y
\end{pmatrix}
\end{pmatrix}=\begin{pmatrix}
F_1\left(S,\begin{pmatrix}
0 \\ 
x
\end{pmatrix}, \begin{pmatrix}
0 \\
y
\end{pmatrix}\right) \\

F_2\left(S,\begin{pmatrix}
0 \\ 
x
\end{pmatrix}, \begin{pmatrix}
0 \\
y
\end{pmatrix}\right) \\
F_3\left(S,\begin{pmatrix}
0 \\ 
x
\end{pmatrix}, \begin{pmatrix}
0 \\
y
\end{pmatrix}\right)
\end{pmatrix}$$
then let
$$X_0:=\overline{D(A)}=\R\times \left(\{0\}\times L^1(0,\infty)\right)^2$$
and its positive cone
$$X_{0+}:=\overline{D(A)}\cap X_+=\R_+\times \left(\{0\}\times L^1_+(0,\infty)\right)^2.$$
We can thus rewrite \eqref{Eq:Model} as the following abstract Cauchy problem:
\begin{equation}\label{Eq:Cauchy_pb}
\left\{
\begin{array}{rcl}
\dfrac{du}{dt}(t)&=&Au(t)+F(u(t)), \forall t\geq 0\\
u(0)&=&u_0\in X_{0}
\end{array}
\right.
\end{equation}
where $u(t):=(S(t),x(t,\cdot), y(t,\cdot))$ and $u_0=(S_0, x_0, y_0)$. 

\subsection{Local existence and positivity}

Using the above semigroup formulation, we can state the classical following result:

\begin{proposition}\label{Prop:Existence}
Suppose that Assumption \ref{Assump1} holds. Then there exists a unique continuous semiflow \break $\{U(t)\}_{t\geq 0}$ on $X_{0+}$ such that for every $z\in X_{0+}$ there exist $t_{\max}\leq \infty$ 
and a continuous map $U\in \Co([0,t_{\max}), X_{0+})$ which is an integrated solution of \eqref{Eq:Cauchy_pb}, \textit{i.e.} such that
$$\int_0^t U(s)z ds\in D(A), \quad \forall t\in[0,t_{\max})$$
and 
$$U(t)z=z+A\int_0^t U(s)z ds+\int_0^t F(U(s)z)ds, \quad \forall t\in[0,t_{\max}).$$
\end{proposition}

\begin{proof}
On one hand, the explicit expression \eqref{Eq:Resolv} of the resolvent of $\hat{A}_k$ for each $k\in\{x,y\}$ ensures us that
$$\left\|(\lambda -A)^{-n}\right\|\leq \dfrac{c}{(\lambda+\mu_0)^n}$$
for some $c>0$ and for every $n\geq 1$, so that $A$ is a Hille-Yosida operator with $(-\mu_0,\infty)\subset \rho(A)$. On the other hand, we can check that the non-linear function $F$ is Lipschitz continuous. Using \cite[Proposition 3.2]{Magal2001} or \cite[Proposition 4.3.3, p. 56]{CazenaveHaraux90} we get the local existence. Now, from \eqref{Eq:Resolv} we deduce that $A$ is resolvent positive, that is to say
$$(\lambda I-A)^{-1}X_+\subset X_+, \quad \forall \lambda\in \rho(A).$$
Moreover, the expression of the non-linearity $F$ implies that for every $r>0$, there exists $c\geq 0$ such that
$$F(z)+c z\in X_+, \quad \forall z\in B(0,r)\cap X_{0+}$$
where $B(0,r)$ denotes the ball of $X$, centred in $0\in X$ and with radius $r$. Finally, using \cite[Proposition 3.6]{Magal2001}, we get the non-negativity of the solution.
\end{proof}

\subsection{Boundedness and global existence}

Let the Banach space
$$\X:=\R\times L^1(0,\infty)\times L^1(0,\infty)$$
endowed with the usual norm and denote by $\X_+$ its positive cone. We are ready to give the main result of this section:

\begin{theorem}\label{Thm:Global}
Suppose that Assumption \ref{Assump1} holds. Then for every $z=(S_0, x_0, y_0)\in \X_+$, there exists a unique mild solution $(S, x, y)\in \Co(\R_+, \X_+)$, that induces a continuous semiflow via:
$$\Phi:\R_+\times \X_+\ni (t,z)\longmapsto \Phi_t(z)=(S(t),x(t,\cdot),y(t,\cdot)).$$
Moreover, the semiflow $\Phi_t=(\Phi_t^S, \Phi_t^x, \Phi_t^y)$ rewrites using Duhamel formulation as follows:
$$\Phi_t(z)=(0,\Phi_t^{x,1}(z),\Phi_t^{y,1}(z))+(\Phi_t^S(z),\Phi_t^{x,2}(z),\Phi_t^{y,2}(z))$$
with $\Phi_t^S(z)>0$ for every $t>0$ and every $z\in \X_+$. Finally, $\Phi_t^x$ and $\Phi_t^y$ are given by:
\begin{equation}\label{Eq:Phi_x1}
\Phi_t^{x,1}(z)(a)=x_0(a-t)e^{-\int_0^t \mu_x(s)ds}\chi_{[t,\infty)}(a),
\end{equation}
\begin{equation}\label{Eq:Phi_x2}
\Phi_t^{x,2}(z)(a)=\Phi_{t-a}^S\\
(z)\int_0^\infty \beta_x(s)\Phi_{t-a}^x(z)(s)ds e^{-\int_0^a \mu_x(u)du} \chi_{[0,t]}(a),
\end{equation}
\begin{equation}\label{Eq:Phi_y1}
\Phi_t^{y,1}(z)(a)=y_0(a-t)e^{-\int_0^t \mu_y(s)ds}\chi_{[t,\infty)}(a),
\end{equation}
\begin{equation}\label{Eq:Phi_y2}
\Phi_t^{y,2}(z)(a)=\Phi_{t-a}^S(z)\int_0^\infty \beta_y(s)\Phi_{t-a}^y(z)(s)ds e^{-\int_0^a \mu_y(u)du} \chi_{[0,t]}(a)
\end{equation}
where $\chi$ denotes the characteristic function. Moreover, there exists a constant $k$ (independent of $z$), such that 
$$\limsup_{t\to \infty(z)} S(t)\leq k, \qquad \limsup_{t\to \infty(z)} x(t,a)\leq k e^{-\mu_0 a}, \qquad \limsup_{t\to \infty(z)} y(t,a)\leq k e^{-\mu_0 a}.$$
\end{theorem}

\begin{proof}
Let $z:=(S_0, x_0, y_0)\in \X_+$ and $(S,x,y)\in \Co([0,t_{\max}), \X_+)$ be the solution of \eqref{Eq:Model}. Suppose by contradiction that $t_{\max}<\infty$. It would imply by \cite[Theorem 3.3]{Magal2001} or \cite[Theorem 4.3.4, p. 57]{CazenaveHaraux90} that
\begin{equation}\label{Eq:Contrad_Expl}\lim_{t\to t_{\max}}(S(t)+\|x(t,\cdot)\|_{L^1}+\|y(t,\cdot)\|_{L^1})=\infty.
\end{equation}
From \eqref{Eq:Model} we see that
$$S'(t)\leq \Lambda -\mu_S S(t)$$
for any $t\geq 0$, which implies that
\begin{equation}\label{Eq:Maj_1}
\limsup_{t\to t_{\max}}S(t)\leq \dfrac{\Lambda}{\mu_S}+\left(S_0-\dfrac{\Lambda}{\mu_S}\right)e^{-\mu_0 t_{\max}}
\end{equation}
by using a Gronwall argument. Now, an integration of \eqref{Eq:Model} leads to
$$\dfrac{d \int_0^\infty x(t,a)da}{dt}=S(t)\int_0^\infty \beta_x(a)x(t,a)da-\int_0^\infty \mu_x(a)x(t,a)da$$
since $x(t,\cdot)\in W^{1,1}(0,\infty)$ and $x(t,a)\xrightarrow[a\to \infty]{}0$ for each $t\in[0,t_{\max}(z))$. Thus we have 
$$S'(t)+\dfrac{d \int_0^\infty x(t,a)da}{dt}\leq \Lambda -\mu_0 \left(S(t)+\int_0^\infty x(t,a)da\right)$$
and then 
\begin{equation}\label{Eq:Maj_2}
\limsup_{t\to t_{\max}(z)} \int_0^\infty x(t,a)da\leq \dfrac{\Lambda}{\mu_0}+\left(S_0+\|x_0\|_{L^1(0,\infty)}-\dfrac{\Lambda}{\mu_0}\right)e^{-\mu_0 t_{\max}}.
\end{equation}
Similarly, we get
\begin{equation}\label{Eq:Maj_3}
\limsup_{t\to t_{\max}(z)} \int_0^\infty y(t,a)da\leq \dfrac{\Lambda}{\mu_0}+\left(S_0+\|y_0\|_{L^1(0,\infty)}-\dfrac{\Lambda}{\mu_0}\right)e^{-\mu_0 t_{\max}}.
\end{equation}
Consequently, we get a contradiction with \eqref{Eq:Cauchy_pb} and then $t_{\max}=\infty$. Finally, from \eqref{Eq:Maj_1}-\eqref{Eq:Maj_2}-\eqref{Eq:Maj_3}, we deduce that the solutions are asymptotically uniformly bounded, since the bound do not depend on the initial condition.
\end{proof}

\section{Equilibria and their stability} \label{Sec:Eq}

As mentioned in the introduction, we have the following result concerning the existence of equilibria:
\begin{proposition}\label{Prop:Eq1}
Suppose that Assumption \ref{Assump1} holds. Then there hold that
\begin{enumerate}
\item if $\max\{R^x_0,R^y_0\} \leq 1$ then there is only one equilibrium that is $E_0$;
\item if $R^x_0 >1\geq R^y_0$ then there are two equilibria: $E_0$ and $E_1$;
\item if $R^x_0 \leq 1<R^y_0$ then there are two equilibria: $E_0$ and $E_2$;
\item if $R^x_0 >1$, $R^y_0 >1$ and $R^x_0 \neq R^y_0 $ then there are three equilibria that are $E_0$, $E_1$ and $E_2$;
\item if $R^x_0 =R^y_0 >1$ then there are an infinite number of equilibria given by $E_0$ and $\{E^*_{\alpha}, \ \alpha\in [1,2]\}$.
\end{enumerate} 
\end{proposition}
We start by reminding the following classical definition
\begin{definition}
Let $S\subset \X_+$ be a subset of $\X_+$ and $E$ be an equilibrium of \eqref{Eq:Model}. Then we say that $E$ is
\begin{itemize}
\item \textbf{(Lyapunov) stable} in $S$ if for every $\varepsilon>0,$ there exists $\eta>0$ such that for every $z\in S$:
$$\|z-E\|_\X \leq \eta \quad \Rightarrow \quad \|\Phi_t(z)-E\|_\X\leq \varepsilon, \quad \forall t\geq 0;$$
\item \textbf{unstable} if $E$ is not stable in $\X_+$;
\item \textbf{locally attractive in $S$} if there exists $\eta>0$ such that for every $z\in S$ satisfying $\|z-E\|_{\X}\leq \eta$, then
\begin{equation}\label{Eq:LocAttrac}\lim_{t\to \infty}\|\Phi_t(z)-E\|_{\X}=0,
\end{equation}
\item \textbf{locally asymptotically stable (L.A.S.) in $S$} if $E$ is stable and locally attractive in $S$;
\item \textbf{globally attractive in $S$} if for every $z\in S$, \eqref{Eq:LocAttrac} is satisfied;
\item \textbf{globally asymptotically stable (G.A.S.) in $S$} if $E^*$ is stable and globally attractive in $S$.
\end{itemize}
\end{definition}
In the following, for notational simplicity, we will not specify the subset $S$ if the latter is the whole positive cone, \textit{i.e.} when $S=\X_+$. We now handle the stability of the equilibria formerly defined.
\begin{proposition}\label{Prop:Stability}
Suppose that Assumption \ref{Assump1} holds. Then the following hold:
\begin{enumerate}
\item if $\max\{R^x_0 , R^y_0 \}<1$ (resp. $>1$) then $E_0$ is L.A.S. (resp. unstable); 
\item if $R^x_0 >\max\{1,R^y_0 \}$ then $E_1$ is L.A.S. If $R_0^y>R_0^x>1$, then $E_1$ is unstable.
\item if $R^y_0 >\max\{1, R^x_0 \}$ then $E_2$ is L.A.S. If $R_0^x>R_0^y>1$, then $E_2$ is unstable;
\item if $R^x_0 =R^y_0 >1$, then for each $\alpha \in [1,2]$, the equilibrium $E^*_\alpha$ is not L.A.S. in $\X_+$.
\end{enumerate}
\end{proposition}

\begin{proof}\mbox{}
Let $E:=(\overline{S},\overline{x}, \overline{y})$ be an equilibrium of \eqref{Eq:Model}, then the linearised system of \eqref{Eq:Model} around $E$ is:
\begin{equation*}
\left\{
\begin{array}{rcl}
\dfrac{du(t)}{dt}&=&Au(t)+DF_E(u(t)), \ \forall t\geq 0, \\
u(0)&=&u_0\in \overline{D(A)}
\end{array}
\right.
\end{equation*}
where $DF_E:X\to X$ denotes the differential of $F$ around $E$ and is defined by:
$$DF_E:\begin{pmatrix}
b_1 \\
b_2 \\
\phi \\
b_3 \\
\psi
\end{pmatrix}=\begin{pmatrix}
-\overline{S}\int_0^\infty \beta_x(a)\phi(a)da-b_1\int_0^\infty \beta_x(a)\overline{x}(a)da-\overline{S}\int_0^\infty \beta_y(a)\psi(a)da-b_1\int_0^\infty \beta_y(a)\overline{y}(a)da \\
\overline{S}\int_0^\infty \beta_x(a)\phi(a)da+b_1\int_0^\infty \beta_x(a)\overline{x}(a)da \\
0 \\
\overline{S}\int_0^\infty \beta_y(a)\psi(a)da+b_1\int_0^\infty \beta_y(a)\overline{y}(a)da \\
0
\end{pmatrix}.$$
Let $A_0$ be the part of $A$ in $\overline{D(A)}$, \textit{i.e.} $A_0:\overline{D(A)}\ni z\longmapsto A_0z:=Az\overline{D(A)}$, then denote by $\{T_{A_0}(t)\}_{t\geq 0}$ the positive semigroup generated by $A_0$. From \eqref{Eq:Resolv}, we know that $(-\mu_0, \infty)\subset \rho(A_0)$ and consequently $s(A_0)\leq -\mu_0<0$ (where $s(A_0)$ is the spectral bound of $A_0$).

Since the semigroup $\{T_{A_0}(t)\}_{t\geq 0}$ is positive, then $\omega_0(\{T_{A_0}(t)\}_{t\geq 0})=s(A_0)$ (where $\omega_0$ denotes the growth bound) by using \cite[Theorem VI. 1.15, p. 358]{EngelNagel2000}. Moreover, we know that  $\omega_{\textnormal{ess}}(\{T_{A_0}(t)\}_{t\geq 0})$, the essential growth bound of $\{T_{A_0}\}_{t\geq 0}$, satisfies $\omega_{\textnormal{ess}}(\{T_{A_0}(t)\}_{t\geq 0})\leq \omega_0(\{T_{A_0}(t)\}_{t\geq 0})$. We then have on one hand:
$$\omega_{\textnormal{ess}}(\{T_{A_0}(t)\}_{t\geq 0})\leq -\mu_0<0.$$
On the other hand, from its above expression, we see that $DF_E(X)$ is finite dimensional, so that $DF_E$ is a compact bounded operator. From \cite[Theorem 1.2]{DucrotMagal2008} we get
$$\omega_{\textnormal{ess}}(\{T_{(A+DF_E)_0}(t)\}_{t\geq 0})=\omega_{\textnormal{ess}}(\{T_{A_0}(t)\}_{t\geq 0})\leq -\mu_0<0$$
where $\{T_{(A+DF_E)_0}(t)\}_{t\geq 0}$ is the $C_0$ semigroup generated by $(A+DF_E)_0$, that is the part of $A+DF_E$ in $\overline{D(A)}$. From \cite[Corollary IV. 2.11, p. 258]{EngelNagel2000}, we deduce that
$$\{\lambda\in \sigma((A+DF_E)_0), \Re(\lambda)\geq -\mu_0 \}$$
is finite and composed (at most) of isolated eigenvalues with finite algebraic multiplicity, where $\sigma(.)$ denotes the spectrum. Consequently, it remains to study the punctual spectrum of $(A+DF_E)_0$. Using \cite[Proposition 4.19, p. 206]{Webb85}, we know that if $s(A+DF_E)<0$ then $E$ is L.A.S., while if $s(A+DF_E)>0$ then $E$ is unstable. We consider exponential solutions, \textit{i.e.} of the form $u(t)=e^{\lambda t}v$, with $0\neq v:=(S,x,y)\in D(A)$ and $\lambda\in \C$. We obtain the following system:
\begin{equation*}
\left\{
\begin{array}{rcl}
\lambda S&=&-\mu_S S-\overline{S}\int_0^\infty \beta_x(a)x(a)da-S\int_0^\infty \beta_x(a)\overline{x}(a)da-\overline{S}\int_0^\infty \beta_y(a)y(a)da-S\int_0^\infty \beta_y(a)\overline{y}(a)da, \\
x'(a)&=&-\mu_x(a)x(a)-\lambda x(a), \\
y'(a)&=&-\mu_y(a)y(a)-\lambda y(a), \\
x(0)&=&\overline{S}\int_0^\infty \beta_x(a)x(a)da+S\int_0^\infty \beta_x(a)\overline{x}(a)da, \\
y(0)&=&\overline{S}\int_0^\infty \beta_y(a)y(a)da+S\int_0^\infty \beta_y(a)\overline{y}(a)da.
\end{array}
\right.
\end{equation*}
We then get
$$x(a)=x(0)\pi_x(a), \qquad y(a)=y(0)\pi_y(a)$$
for every $a\geq 0$,
\begin{flalign}\label{Eq:Charac1}
&S\left(\lambda+\mu_S+\int_0^\infty \beta_x(a)\overline{x}(a)da+\int_0^\infty \beta_y(a)\overline{y}(a)da\right) \nonumber \\
=&-\overline{S}\left(x(0)\int_0^\infty \beta_x(a)\pi_x(a)e^{-\lambda a}da+y(0)\int_0^\infty \beta_y(a)\pi_y(a)e^{-\lambda a}da\right)
\end{flalign}
and
\begin{equation}\label{Eq:Charac2}
\left\{
\begin{array}{rcl}
x(0)\left(1-\overline{S}\int_0^\infty \beta_x(a)\pi_x(a)e^{-\lambda a}da\right)&=&S\int_0^\infty \beta_x(a)\overline{x}(a)da, \\
y(0)\left(1-\overline{S}\int_0^\infty \beta_y(a)\pi_y(a)e^{-\lambda a}da\right)&=&S\int_0^\infty \beta_y(a)\overline{y}(a)da,
\end{array}
\right.
\end{equation}
with $(S,x(0),y(0))\neq (0,0,0)$.
\begin{enumerate}
\item Let $E:=E_0$. From \eqref{Eq:Charac1}-\eqref{Eq:Charac2}, we get:
\begin{flalign*}
S^*_0\left(x(0)\int_0^\infty \beta_x(a)\pi_x(a)e^{-\lambda a}da+y(0)\int_0^\infty \beta_y(a)\pi_y(a)e^{-\lambda a}da\right)&=-S\left(\lambda+\mu_S\right)\\
x(0)\left(1-S^*_0\int_0^\infty \beta_x(a)\pi_x(a)e^{-\lambda a}da\right)&=0, \\
y(0)\left(1-S^*_0\int_0^\infty \beta_y(a)\pi_y(a)e^{-\lambda a}da\right)&=0,
\end{flalign*}
Suppose first that $\max\{R_0^x, R_0^y\}>1$. Without loss of generality, we can suppose that $R_0^x>1$. We see that the function
$$f:\R\ni\lambda \longmapsto S^*_0\int_0^\infty \beta_x(a)\pi_x(a)e^{-\lambda a}da\in \R$$
is strictly decreasing, with $f(0)=R_0^x>1$. We see that there exists $\lambda^*>0$ such that $f(\lambda^*)=1$, and considering \textit{e.g.}
$$(S,x(0),y(0))=\left(\dfrac{S^*_0\int_0^{\infty}\beta_x(a)\pi_x(a)e^{-\lambda^* a}da}{\lambda^*+\mu_S},1,0\right)$$
we deduce that $s(A+DF_{E_0})\geq \lambda^*>0$ so $E_0$ is unstable. Suppose now that $\max\{R_0^x, R_0^y\}<1$ and that there exists $\lambda\in \sigma(A+DF_{E_0})$ such that $\Re(\lambda)\geq 0$. If $y(0)\neq 0$, then we have
$$1=S^*_0\int_0^\infty \beta_y(a)\pi_y(a)e^{-\lambda a}da\leq R_0^y<1.$$
so we have $y(0)=0$. Likewise we deduce that $x(0)=0$, but it then follows that $S=0$, which is absurd. Consequently $E_0$ is L.A.S.
\item Let $E:=E_1$. From \eqref{Eq:Charac1}-\eqref{Eq:Charac2}, we get:
\begin{flalign*}
S^*_1\left(x(0)\int_0^\infty \beta_x(a)\pi_x(a)e^{-\lambda a}da+y(0)\int_0^\infty \beta_y(a)\pi_y(a)e^{-\lambda a}da\right)&=-S\left(\lambda+\mu_S R_0^x\right)\\
x(0)\left(1-S^*_1\int_0^\infty \beta_x(a)\pi_x(a)e^{-\lambda a}da\right)&=S\mu_S(R_0^x-1), \\
y(0)\left(1-S^*_1\int_0^\infty \beta_y(a)\pi_y(a)e^{-\lambda a}da\right)&=0.
\end{flalign*}
Suppose that $R_0^y>R_0^x>1$, then when $y(0)\neq 0$, we obtain
$$S^*_1\int_0^\infty \beta_y(a)\pi_y(a)e^{-\lambda a}da=1.$$
We see that the function
$$f:\R\ni\lambda \longmapsto S^*_1\int_0^\infty \beta_y(a)\pi_y(a)e^{-\lambda a}da$$
is strictly decreasing, with $f(0)=S^*_1 r_y=\frac{r_x}{r_y}>1$, and we deduce that $s(A+DE_{E_1})>0$ so $E_1$ is unstable. Suppose now that $R_0^x>\max\{R_0^y,1\}$ and that $\lambda\in \sigma(A+DE_{E_1})$ with $\Re(\lambda)\geq 0$. If $y(0)\neq 0$ then
$$S^*_1\int_0^\infty \beta_y(a)\pi_y(a)e^{-\lambda a}da\leq  \dfrac{r_y}{r_x}<1$$
which is absurd, so $y(0)=0$. We deduce that
\begin{flalign*}
S^*_1\left(x(0)\int_0^\infty \beta_x(a)\pi_x(a)e^{-\lambda a}da\right)&=-S\left(\lambda+\mu_S R_0^x\right)\\
x(0)\left(1-S^*_1\int_0^\infty \beta_x(a)\pi_x(a)e^{-\lambda a}da\right)&=S\mu_S(R_0^x-1)
\end{flalign*}
whence
\begin{flalign*}
1&=S^*_1\int_0^\infty \beta_x(a)\pi_x(a)e^{-\lambda a}da-\dfrac{S_1^* \mu_S (R_0^x-1)\int_0^\infty \beta_x(a)\pi_x(a)e^{-\lambda a}da}{\lambda+\mu_S R_0^x} \\
&=S^*_1\int_0^\infty \beta_x(a)\pi_x(a)e^{-\lambda a}da\left(\dfrac{\lambda +\mu_S}{\lambda +\mu_S R_0^x}\right)
\end{flalign*}
so
$$\dfrac{\lambda +\mu_S R_0^x}{\lambda +\mu_S}=S^*_1\int_0^\infty \beta_x(a)\pi_x(a)e^{-\lambda a}da.$$
Considering real and imaginary parts of $\lambda$, we get:
\begin{flalign*}
&\dfrac{\left((\Re(\lambda)+\mu_S R_0^x)+i\Im(\lambda)\right)\left((\Re(\lambda)+\mu_S)-i \Im(\lambda)\right)}{\left(\Re(\lambda)+\mu_S\right)^2+\Im(\lambda)^2}\\
&=S^*_1\int_0^\infty \beta_x(a)\pi_x(a)e^{-\Re(\lambda)a}\left(\cos(a\Im(\lambda))+i\sin(a\Im(\lambda))\right)da
\end{flalign*}
then identifying the real part, we obtain:
\begin{flalign*}
&(\Re(\lambda)+\mu_S R_0^x)(\Re(\lambda)+\mu_S)+\Im(\lambda)^2
\\
&=\left((\Re(\lambda)+\mu_S)^2+\Im(\lambda)\right)S^*_1 \int_0^\infty \beta_x(a)\pi_x(a)e^{-\Re(\lambda)a}\cos(a\Im(\lambda))da.
\end{flalign*}
It follows that
\begin{flalign*}0<\mu_S(R_0^x-1)(\Re(\lambda)+\mu_S)&=\left((\Re(\lambda)+\mu_S)^2+\Im(\lambda)^2\right)\left(S^*_1 \int_0^\infty \beta_x(a)\pi_x(a)e^{-\Re(\lambda)a}\cos(a \Im(\lambda))da-1\right) \\
&\leq \left((\Re(\lambda)+\mu_S)^2+\Im(\lambda)^2\right)\left(\dfrac{r_y}{r_x}-1\right)\leq 0
\end{flalign*}
since $\Re(\lambda)\geq 0$. We deduce that $s(A+DE_{E_1})<0$ and $E_1$ is L.A.S.
\item Similar arguments as for the latter point allow us to prove the result for $E_2$.
\item Suppose that $R_0^x=R_0^y>1$. Let $\alpha \in[1,2]$. Since the set of equilibria $\{E^*_a, a\in[1,2]\}$ is compact, we can prove that for every $\eta>0$ there exists $\tilde{\alpha}\in[1,2]\setminus\{\alpha\}$ such that $\|E^*_{\tilde{\alpha}}-E^*_\alpha\|_{\X}\leq \eta$. Moreover, by definition of $E^*_{\tilde{\alpha}}$, we have 
$$\lim_{t\to \infty} \|\Phi_t(E^*_{\tilde{\alpha}})-E^*_\alpha\|_{\X}=\|E^*_{\tilde{\alpha}}-E^*_\alpha\|_{\X}>0$$
hence $E^*_\alpha$ is not locally attractive, and therefore not L.A.S.
\end{enumerate}
\end{proof}

\section{Compact attractor and basins of attraction} \label{Sec:Attractor}

\subsection{Preliminaries }

In the sequel, we will denote by $\OO_z=\{\Phi_t(z), t\geq 0\}$ the orbit starting from $z\in \X_+$ and
$$\omega(z)=\underset{\tau\geq 0}{\bigcap} \overline{\{\Phi_t(z), t\geq \tau\}}$$
the $\omega$-limit set of $z$. We follow \cite[Section 3]{Perasso2019}, to prove the existence of a compact attractor.

\begin{lemma}
For every $z\in \X_+$, the orbit $\OO_z\subset \X_+$ is relatively compact, \textit{i.e.} $\overline{\OO_z}$ is compact.
\end{lemma}

\begin{proof}
Define the ball $B_r:=\{\tilde{z}\in \X, \|\tilde{z}\|_{\X}\leq r\}$ for any $r>0$. From \eqref{Eq:Phi_x1}-\eqref{Eq:Phi_y1}, we see that for every $r>0$ and every $z\in \X_+\cap B_r$, we have
$$\left\|\left(0,\Phi_t^{x,1}, \Phi_t^{y,1}\right)\right\|_{\X}\leq 2r e^{-\mu_0 t}, \quad \forall t\geq 0.$$
Moreover we can prove that for any $t\geq 0$, $(\Phi_t^S, \Phi_t^{x,2}, \Phi_t^{y,2})$ maps bounded sets of $\X_+$ into relatively compact sets in $\X$. Indeed, let $M\subset \X_+$ be a bounded subset of $\X$, \textit{i.e.} there exists $r>0$ such that $\|z\|_{\X}\leq r$ for any $z\in M$. First, we see that $\Phi_t^S(M)$ is relatively compact since it is finite dimensional. Moreover, from Theorem \ref{Thm:Global}, we deduce that for every $t\geq 0$, there exists a constant $c(r)>0$ such that
$$\int_0^\infty \Phi_{t-a}^x(z)(s)ds\leq c, \qquad \int_0^\infty \Phi_{t-a}^y(z)(s)ds\leq c$$
for any $a\in[0,t]$ and every $z\in M$. Using \eqref{Eq:Phi_x2}-\eqref{Eq:Phi_y2}, we deduce that
$$\Phi_t^{x,2}(z)(a)\leq r\|\beta_x\|_{L^\infty}c e^{-\mu_0 a}\chi_{[t,\infty)}(a), \qquad \Phi_t^{y,2}(z)(a)\leq r\|\beta_y\|_{L^\infty}c e^{-\mu_0 a}\chi_{[t,\infty)}(a)$$
for any $(t,z)\in \R_+\times M$. Finally, the Fréchet-Kolmogorov theorem ensures that the sets $\Phi_t^{x,2}(M)$ and $\Phi_t^{y,2}(M)$ are relatively compact. From \cite[Proposition 3.1.3 p. 100]{Webb85}, we deduce that for every $z\in \X_+$, the orbit $\OO_x$ is relatively compact.
\end{proof}

The latter compactness result of the orbits then leads to the existence of a compact attractor in the following sense (see \textit{e.g.} \cite[Lemma 3.1.1 and 3.1.2, p. 36]{Hale88}, or \cite[Theorem 4.1, p. 167]{Walker80}).

\begin{lemma}\label{Lemme:OmegaLimit}
For every $z\in \X_+$,
\begin{enumerate}
\item $\omega(z)$ is non-empty, compact and connected;
\item $\omega(z)$ is invariant under $\Phi_t$, \textit{i.e.} $\Phi_t(\omega(z))=\omega(z)$;
\item $\omega(z)$ is an attractor, \textit{i.e.} $\lim_{t\to \infty} d(\Phi_t(z), \omega(z))=0$. 
\end{enumerate}
\end{lemma}
We remind the following classical result and we give its proof for completeness.
\begin{lemma}\label{Lemma:Semigroup}
Let $c\in(0,\infty)$, $k\in(0,\infty)$ and $u$ be the solution of the PDE:
\begin{equation*}
\left\{
\begin{array}{rcl}
\dfrac{\partial v(t,a)}{\partial t}+\dfrac{\partial v(t,a)}{\partial a}&=&-\mu_x(a)v(t,a), \\
v(t,0)&=&k\displaystyle \int_0^c \beta_x(a)v(t,a)da, \\
v(0,a)&=&v_0(a)
\end{array}
\right.
\end{equation*}
for every $a\in (0,c)$ and every $t\geq 0$. Suppose that $u_0\in L^1_+(0,c)\setminus \{0\}$ and that
$$k\int_0^c \beta_x(a)\pi_x(a)da>1$$
then
$$\lim_{t\to \infty}\int_0^c v(t,a)da=\infty.$$
\end{lemma}
The same holds when replacing $x$ by $y$.

\begin{proof}
Define the linear operator $\A:D(\A)\subset L^1(0,c)\to L^1(0,c)$ by
$$\A u=-u'-\mu_x u$$
with
$$D(\A):=\left\{u\in W^{1,1}(0,c), u(0)=k\int_0^c \beta_x(a)u(a)da\right\}$$
where $W^{1,1}(0,c):=\left\{u\in L^1(0,c), u'\in L^1(0,c)\right\}$ is the Sobolev space. It is classical that $\A$ generates a positive $C_0$-semigroup $\{T_\A(t)\}_{t\geq 0}$. Since $[0,c]$ is compact, then $\A$ has a compact resolvent, and consequently the spectrum of $\A$ is composed at most of isolated eigenvalues with finite algebraic multiplicity. This follows from the fact that the canonical injection $i:(D(\A),\|\cdot\|_{D(\A)}\to (L^1(0,c), \|\cdot\|_{L^1(0,c)})$ is compact by the Rellich-Kondrachov Theorem. Any eigenvalue of $\A$ has to satisfy:
$$u'+\lambda u+\mu_x u=0$$
where $u\in D(\A)$. We hence get the following characteristic equation:
$$1=k\int_0^c \beta_x(a)e^{-\lambda a}\pi_x(a)da$$
which is satisfied for some $\lambda>0$ by definition of $\ep$. Now we prove that $(\lambda-\A)^{-1}$ is positivity improving for $\lambda$ large enough, \textit{i.e.} $(\lambda-\A)^{-1}h(s)>0$ a.e. $s\in[0,c]$ for any $h\in L^1_+(0,c)\setminus \{0\}$. Let $\nu>0$, $h\in L^1_+(0,c)\setminus\{0\}$ and $u=(\lambda-\A)^{-1}h$. Then we have
$$u'+\lambda u+\mu_x u=h$$
with $u\in D(\A)$, \textit{i.e.}
\begin{flalign*}
u(a)&=u(0)e^{-\lambda a-\int_0^a\mu_x(s)ds}+\int_0^a h(s)e^{-\lambda(a-s)-\int_s^a \mu(\xi)d\xi}ds \\
&=k\left(\int_0^c \beta_x(s)u(s)ds\right) e^{-\lambda a-\int_0^a\mu_x(s)ds}+\int_0^a h(s)e^{-\int_s^a (\lambda+\mu_x(\xi))d\xi}ds
\end{flalign*}
and for $\lambda>0$ large enough, we get
$$\left(1-k\int_0^c \beta_x(a)e^{-\lambda a}\pi_x(a)da\right)\int_0^c \beta_x(a)u(a)da=\int_0^c \beta_x(a)\int_0^a h(s)e^{-\int_s^a (\lambda+\mu_x(\xi))d\xi}dsda.$$
We see that
$$\int_0^c \beta_x(a)\int_0^a h(s)e^{-\lambda(a-s)-\int_s^a \mu_x(\xi)d\xi}dsda>0$$
whence $\int_0^c \beta_x(a)u(a)da>0$ and $u(a)>0$ for every $a\in[0,c]$. We deduce that $(\lambda-\A)^{-1}$ is positivity improving. Using \cite[p. 165]{Clement87}, we deduce that $\{T_\A(t)\}_{t\geq 0}$ is irreducible, \textit{i.e.} for any $\phi\in L^1_+(0,\infty)\setminus \{0\}$ and any $\psi\in L^\infty_+(0,\infty)\setminus\{0\}$, there exists $t>0$ such that $\left\langle T_{\A}(t)\phi, \psi \right \rangle>0$, where $\left\langle \cdot, \cdot \right\rangle$ denotes the duality pairing between $L^1$ and $L^\infty$. Since the semigroup is positive, we know that
$$\omega_0(\{T_\A(t)\}_{t\geq 0})=s(\A)>0.$$
Moreover, since the spectrum of $\A$ is punctual, then
$$\omega_{\textnormal{ess}}\left(\{T_\A(t)\}_{t\geq 0}\right)=-\infty.$$
Consequently $\{T_\A(t)\}_{t\geq 0}$ is both irreducible and has a spectral gap (\textit{i.e.} $\omega_0>\omega_{\textnormal{ess}}$). On one hand we know that $s(\A)$ is a simple pole of the resolvent of $\A$, with geometric multiplicity equal to one (see e.g. \cite[p. 224]{Clement87}). On the other hand, consequently to \cite[Theorem 9.11. p. 224]{Clement87} we get
$$\lim_{t\to \infty} \left\|e^{-s(\A) t}T_\A(t)f-Pf\right\|_{L^1(0,c)}=0$$
for every $f\in L^1(0,c)$, where $P$ is the projection on $\Ker(\lambda-\A)$ along $\Rg(\lambda-\A)$, that is an operator of rank one and positivity improving. Since $v_0\in L^1_+(0,c)\setminus \{0\}$, then we deduce that $P v_0(a)>0$ for a.e. $a\in(0,c)$ and that $\lim_{t\to \infty}e^{-s(\A)t}=0$. Thus we obtain
$$\lim_{t\to \infty}\int_0^c T_\A(t)v_0(a)da=\lim_{t\to \infty} \int_0^c v(t,a)da=\infty.$$
\end{proof}

\subsection{Basins of attraction}

We now give some results about the attractive sets, depending on the initial condition as well as the thresholds $R_0^x$ and $R_0^y$. 
\begin{proposition} \label{Prop:Attract}
Suppose that Assumptions \ref{Assump1} holds, then:
\begin{enumerate}
\item the sets $\partial \S_x$ and $\partial \S_y$ are positively invariant, \textit{i.e.} $\Phi_t(\partial \S_x)\subset \partial \S_x$ and $\Phi_t(\partial \S_y)\subset \partial \S_y$, $\forall t\geq 0$. Moreover, for every $z:=(x_0,y_0,z_0)\in\partial \S_x$ (respectively $z\in \partial \S_y$), then
\begin{equation}\label{Eq:Phi-Decreas}
\|\Phi_t^x(z)\|_{L^1(\R_+)}\leq \|x_0\|_{L^1(\R_+)} e^{-\mu_0 t}, \quad \left(\textnormal{resp. } 
\|\Phi_t^y(z)\|_{L^1(\R_+)}\leq \|y_0\|_{L^1(\R_+)} e^{-\mu_0 t}\right)
\end{equation}
for every $t\geq 0$;

\item the equilibrium $E_0$ is globally exponentially stable for $\Phi_t$ restricted to $\partial \S_x\cap \partial \S_y$;

\item there exists $c>0$ such that for every $z\in \X_+$ we have:
$$\liminf_{t\to\infty}\Phi_t^S(z)\geq c;$$

\item for every $z\in \S_x$ (resp. $z\in \S_y$), there exists $\tau\geq 0$ such that
$$\int_0^\infty \beta_x(a)\Phi^x_t(z)(a)da>0 \qquad (\textnormal{resp. } \int_0^\infty \beta_y(a)\Phi^y_t(z)(a)da>0)$$
for every $t\geq \tau$. Moreover, the sets $\S_x$ and $\S_y$ are asymptotically positively invariant, \textit{i.e.} for every $z\in \S_x$ (resp. $z\in \S_y$), there exists $\tau\geq 0$ such that $\Phi_t(z)\in \S_x$ (resp. $\Phi_t(z)\in \S_y$) for every $t\geq \tau$.

\item Let $z\in \partial \S_x$, then $\|\Phi_t^x(z)\|_{L^1(0,\infty)}\leq e^{-\mu_0 t} \|\Phi_0^x(z)\|_{L^1(0,\infty)}$ for every $t\geq 0$. Moreover, there hold:
\begin{enumerate}
\item if $R^y_0 >1$ and $z\in \S_y$ then $\omega(z)\subset \partial \S_x\cap\S_y$ and $\lim_{t\to \infty}\|\Phi_t(z)-E_2\|_{\X}=0$;
\item if $R^y_0\leq 1$ then $\lim_{t\to \infty}\|\Phi_t(z)-E_0\|_{\X}=0$.
\end{enumerate}

\item Let $z\in \partial \S_y$, then $\|\Phi_t^y(z)\|_{L^1(0,\infty)}\leq e^{-\mu_0 t}\|\Phi_0^y(z)\|_{L^1(0,\infty)}$ for every $t\geq 0$. Moreover, there hold:
\begin{enumerate}
\item if $R^x_0 >1$ and $z\in \S_x$ then $\omega(z)\subset \S_x\cap \partial \S_y$ and $\lim_{t\to \infty}\|\Phi_t(z)-E_1\|_{\X}=0$;
\item if $R^x_0\leq 1$ then $\lim_{t\to \infty}\|\Phi_t(z)-E_0\|_{\X}=0$.
\end{enumerate}

\item Let $z\in \X_+$:
\begin{enumerate}
\item if $R_0^x\leq 1$, then $\omega(z)\subset \partial \S_x$ and $\lim_{t\to \infty}\|\Phi_t^x(z)\|_{L^1(0,\infty)}=0$;
\item if $R_0^y\leq 1$, then $\omega(z)\subset \partial \S_y$ and $\lim_{t\to \infty}\|\Phi_t^y(z)\|_{L^1(0,\infty)}=0$;
\end{enumerate}
\item Let $z\in \S_x\cap \S_y$:
\begin{enumerate}
\item if $\max\{R^x_0 , R^y_0 \}>1$, then $\omega(z) \subset \S_x\cup \S_y$; 
\item if $R_0^x>\max\{1,R_0^y\}$, then $\omega(z)\subset \S_x$;
\item if $R_0^y>\max\{1,R_0^x\}$, then $\omega(z)\subset \S_y$.
\end{enumerate}
\end{enumerate}
\end{proposition}

\begin{proof}\mbox{}
\begin{enumerate}
\item Let $z\in \partial \S_x$. We remind that the component in $x$ of the semiflow rewrites as $\Phi_t^x(z)(a)=\Phi_t^{x,1}(z)(a)+\Phi_t^{x,2}(z)(a)$, where $\Phi_t^{x,1}(z)$ and $\Phi_t^{x,2}(z)$ are respectively defined in \eqref{Eq:Phi_x1} and \eqref{Eq:Phi_x2}. We see that
$$\int_0^{\overline{\beta_x}}\Phi_t^{x,1}(z)(a)da\leq \int_0^{\overline{\beta_x}}x_0(a)da=0$$
for every $t\geq 0$, which implies
$$\int_0^{\infty}\beta_x(a)\phi_t^{x,1}(z)(a)da=\int_0^{\overline{\beta_x}}\beta_x(a)\Phi_t^{x,1}(z)(a)da\leq \|\beta_x\|_{L^\infty} \int_0^{\overline{\beta_x}}\Phi_t^{x,1}(z)(a)da=0$$
for every $t\geq 0$. Hence we deduce that the function $F(t)=\int_0^\infty \beta_x(a)\Phi_t^x(z)(a)da$ satisfies
$$F(t)\leq \|\beta_x\|_{L^\infty} \int_0^t F(t-s)\Phi_{t-a}^S(z)da.$$
Then a Gronwall argument states that $F(t)=0$ for every $t\geq 0$ and we deduce from \eqref{Eq:Phi_x2} that $\Phi_t^{x,2}(z)(a)=0$ for every $t\geq 0$ and every $a\geq 0$. Consequently we get
$$\int_0^{\overline{\beta_x}} \Phi_t^x(z)(a)da=\int_0^{\overline{\beta_x}} \Phi_t^{x,1}(z)(a)da+\int_0^{\overline{\beta_x}} \Phi_t^{x,2}(z)(a)da=0$$
for every $t\geq 0$, thus $\partial \S_x$ is positively invariant. Moreover, we can deduce that
$$\int_0^\infty \Phi^x_t(z)(a)da=\int_0^{\overline{\beta_x}}\Phi^{x,1}_t(z)(a)da\leq e^{-\mu_0 t}\|x_0\|_{L^1(\R_+)}$$
for every $t\geq 0$ by using \eqref{Eq:Phi_x1} and Assumption \ref{Assump1}. Similar arguments would prove on one hand that $\partial \S_y$ is positively invariant, and on the other hand that \eqref{Eq:Phi-Decreas} holds for every $z:=(x_0,y_0,z_0)\in \partial \S_y$ and every $t\geq 0$ by using \eqref{Eq:Phi_y2} and Assumption \ref{Assump1}.
\item Let $z:=(x_0,y_0,z_0)\in \partial \S_x\cap \partial \S_y$. Using the first point, we have
$$\int_0^\infty \beta_x(a)\Phi_t^x(z)(a)da=0, \qquad \int_0^\infty \beta_y(a)\Phi_t^y(z)(a)da=0$$
for every $t\geq 0$. Consequently, from Problem \eqref{Eq:Model} we get
$$\Phi_t^S(z)=S_0 e^{-\mu_S t}+\dfrac{\Lambda}{\mu_S}(1-e^{-\mu_S t})$$
for every $t\geq 0$. Using \eqref{Eq:Phi-Decreas}, we deduce
$$\|\Phi_t(z)-E_0\|_{\X}\leq e^{-\mu_0 t}\left(\left|S_0-S^*_0\right|+\|x_0\|_{L^1(\R_+)}+\|y_0\|_{L^1(\R_+)}\right)=e^{-\mu_0 t}\left\|z-E_0\right\|_{\X}$$
which proves the second point.
\item Let $z\in \X_+$ and let $(S,x,y)\in \Co(\R_+, \X_+)$ the solution of \eqref{Eq:Model}. By Theorem \ref{Thm:Global}, we know that there exists $k>0$ (independent of $z$) such that
$$\limsup_{t\to \infty} \int_0^\infty \beta_x(a)x(t,a)da\leq k, \qquad \limsup_{t\to \infty} \int_0^\infty \beta_y(a)y(t,a)da\leq k.$$
Injecting the latter equation into \eqref{Eq:Model} implies that for every $\ep>0$, there exists $t_0>0$ such that
$$S'(t)\geq \Lambda-\mu_S S(t)-2S(t)(k+\ep)$$
for every $t\geq t_0$, whence
$$\liminf_{t\to \infty} S(t)\geq \dfrac{\Lambda}{\mu_S+2(k+\ep)}>0$$
whence the third point.
\item Let $z\in \S_x$, then there exists $0\leq b_1<b_2\leq \overline{\beta_x}$ such that 
$$\int_{b_1}^{b_2} x_0(a)da>0.$$
By Assumption \ref{Assump1}, we may find $c\in(\underline{\beta_x},\infty)$ such that $\beta_x(a)>0$ a.e. $a\in [\underline{\beta_x},c)$. Let $t_0=c-b_2$, then using \eqref{Eq:Phi_x1}, we see that
\begin{flalign*}
\int_{\underline{\beta_x}}^{c}\Phi_{t_0}^x(z)(a)da&\geq e^{-t_0 \|\mu_x\|_{L^\infty}} \int_{\underline{\beta_x}}^{c}x_0(a-t)\chi_{[t,\infty)}(a)da \\
&\geq e^{-t_0 \|\mu_x\|_{L^\infty}} \int_{b_1}^{b_2}x_0(a)da>0
\end{flalign*}
Since $\beta_x>0$ a.e. on $[\underline{\beta_x}, c]$ and $\Phi^S_t(z)>0$ for every $t>0$ due to Theorem \ref{Thm:Global}, then we get
$$\Phi_{t_0}^x(z)(0)>0$$
by using \eqref{Eq:Phi_x2}. By continuity arguments, there exists $t_1>t_0$ such that 
$$\Phi_{t}^x(z)(0)>0, \qquad \forall t\in[t_0,t_1).$$
Let $\Delta_t=t_1-t_0>0$ and let $0<\ep<<c-\underline{\beta_x}$, then we see that
\begin{flalign*}
\int_{\underline{\beta_x}}^{c}\Phi_{t_0+s}^x(z)(a)da&\geq e^{-s \|\mu_x\|_{L^\infty}}  \int_{\underline{\beta_x}}^{c}\Phi_{t_0+s-a}^x(z)(0) \chi_{[0,t_0+s]}(a)da>0
\end{flalign*}
for any $s\in[\underline{\beta_x}+\ep, t_1-t_0+c-\ep)$. Consequently we have
$$\Phi_t^x(z)(0)>0, \qquad \forall t\in[t_0+\underline{\beta_x}+\ep, t_1+c-\ep].$$
Similarly we can prove that
$$\Phi_t^x(z)(0)>0, \qquad \forall t\in\left[t_0+n\left(\underline{\beta_x}+\ep\right), t_1+n\left(c-\ep\right)\right]$$
for any $n\in \N$. Since $n(c-\underline{\beta_x}-2\ep)\xrightarrow[n \to \infty]{}\infty$, we deduce that there exists $t^*>0$ such that
$$\Phi_t^x(z)(0)>0, \qquad \forall t\in[t^*, t^*+c].$$
Let $0<\ep<c$ and $t=t^*+c+\ep$, then we get
$$\Phi_t^x(z)(0)\geq \Phi_t^S(z)\int_{\underline{\beta_x}}^{c-\ep}\beta_x(a)\Phi_{t-a}^x(z)(0)da>0.$$
Hence we deduce that
$$\Phi_t^x(z)(0)>0, \qquad \forall t\in[t^*,t^*+2c]$$
then repeating this argument we obtain
$$\int_0^\infty \beta_x(a)\Phi_t^x(z)(a)da=\Phi_t^x(z)(0)>0, \qquad \forall t\geq t^*.$$
Finally, we obtain
$$\int_0^{\overline{\beta_x}} \Phi_t^x(z)(a)da \geq e^{-\overline{\beta_x}\|\mu_x\|_{L^\infty}} \int_0^{\overline{\beta_x}}\Phi_{t-a}^x(z)(0)da>0$$
for every $t>t^*$, so that $\S_x$ is asymptotically positively invariant. The same arguments would prove the result for $y$.

\item Let $z:=(S_0, x_0, y_0)\in \partial \S_x$ and $(S,x,y)\in \Co(\R_+,\X_+)$ be the solution of \eqref{Eq:Model}. Since $\partial \S_x$ is positively invariant by the first point, then $\omega(z)\subset \partial \S_x$. Consequently we have
$$\int_0^\infty \beta_x(a)x(t,a)da=0$$
for every $t\geq 0$ and from \eqref{Eq:Phi_x1}-\eqref{Eq:Phi_x2} we get
$$\|\Phi_t^x(z)\|_{L^1(0,\infty)}\leq \|x_0\|_{L^1(0,\infty)}e^{-\mu_0 t}\xrightarrow[t\to \infty]{}0.$$
We deduce that $(S,y)$ satisfies \eqref{Eq:Model_1pop}. If $R_0^y>1$ and $z\in \S_y$, then from Proposition \ref{Prop:Magal} we obtain
$$\lim_{t\to \infty}\|\Phi_t(z)-E_2\|_{\X}\leq
\lim_{t\to \infty}\left(\left\|(\Phi_t^S(z),\Phi_t^y(z))-(S^*_2,y^*_2)\right\|_{\R\times L^1(0,\infty)}+\|\Phi_t^x(z)\|_{L^1(0,\infty)}\right)=0$$
whence $\omega(z)\subset \S_y$. If $R_0^y\leq 1$, we deduce from Proposition \ref{Prop:Magal} that 
$$\lim_{t\to \infty}\|\Phi_t(z)-E_0\|_{\X}\leq 
\lim_{t\to \infty}\left(\left\|(\Phi_t^S(z),\Phi_t^y)-\left(S^*_0,0\right)\right\|_{\R\times L^1(0,\infty)}+\left\|\Phi_t^x(z)\right\|_{L^1(0,\infty)}\right)=0.$$

\item The latter arguments would prove the case $z\in \partial \S_y$. 
\item Let $z\in \X_+$: \textbf{(a)} Suppose that $R_0^x\leq 1$. A simple upper bound on \eqref{Eq:Model} leads to
\begin{equation*}
\left\{
\begin{array}{rcl}
S'(t)&\leq&\Lambda-\mu_S S(t)-S(t) \int_0^\infty\beta_x(a)x(t,a)da, \\
\dfrac{\partial x(t,a)}{\partial t}+\dfrac{\partial x(t,a)}{\partial a}&=&-\mu_x(a)x(t,a), \\
x(t,0)&=&S(t) \int_0^\infty \beta_x(a)x(t,a)da.
\end{array}
\right.
\end{equation*}
since $\Phi_t^y(z)(a)\geq 0$ a.e. $a\geq 0$. From Proposition \ref{Prop:Magal}, we obtain
$$\lim_{t\to \infty}\|\Phi_t^x(z)\|_{L^1(\R_+)}=0$$
since $R_0^x\leq 1$, whence $\omega(z)\subset \partial \S_x$. \\
\textbf{(b)} The same argument proves the result when $R_0^y\leq 1$.

\item Let $z\in \S_x\cap \S_y$:
\textbf{(a)} Suppose that $\max\{R_0^x, R_0^y\}>1$. Without loss of generality, we can suppose that $R_0^y>1$. By continuity arguments, there exists $\overline{\beta_y}<c<\infty$ such that
$$\dfrac{\Lambda \int_0^c \beta_y(a)\pi_y(a)da}{\mu_S}>1$$
and there exists $\ep>0$ small enough such that
\begin{equation}\label{Eq:R0y_small}
\dfrac{\Lambda \int_0^c \beta_y(a)\pi_y(a)da}{\mu_S+\ep(\|\beta_x\|_{L^\infty}+\|\beta_y\|_{L^\infty})}>1.
\end{equation}
Let $\M_\ep:=\{z\in \S_x\cap \S_y, \|z-E_0\|_{\X}\leq \ep\}$. We first prove that for every $z\in \M_\ep$, there exists $\overline{t}(z)$ such that
\begin{equation}\label{Eq:Unstab}
\|\Phi_{\overline{t}}(z)-E_0\|_{\X}>\ep
\end{equation}
holds. By contradiction, suppose that there exists $z:=(S_0, x_0, y_0)\in \S_x\cap \S_y$ such that
\begin{equation}\label{Eq:Unstab_Contrad}
\|\Phi_t(z)-E_0\|_{\X}\leq \ep, \qquad \forall t\geq 0.
\end{equation}
We know by Proposition \ref{Prop:Attract} (4) that there exists $\tau\geq 0$ such that $\Phi_t(z)\in \S_x\cap \S_y$ for every $t\geq \tau$. Thus, a Gronwall argument leads to
$$S(t)\geq \dfrac{\Lambda}{\mu_S+\ep\left(\|\beta_y\|_{L^\infty}+\|\beta_x\|_{L^\infty}\right)}, \qquad \forall t\geq 0.$$
Now, we denote for convenience $y(t,a)=\Phi_t^y(z)(a)$, and we deduce from \eqref{Eq:Model} that $y$ satisfies the following system:
\begin{equation*}
\left\{
\begin{array}{rcl}
\dfrac{\partial y(t,a)}{\partial t}+\dfrac{\partial y(t,a)}{\partial a}&=&-\mu_y(a)y(t,a), \\
 y(t,0)&\geq& \dfrac{\Lambda}{\mu_S +\ep\left(\|\beta_x\|_{L^\infty}+\|\beta_y\|_{L^\infty}\right)}\displaystyle\int_0^c \beta_y(a)y(t,a)da,  \\
y(\tau,a)&=&\Phi_\tau^y(z)(a)
\end{array}
\right.
\end{equation*}
for a.e. $a\in [0,c]$ and every $t\geq \tau$. We then have $y(t,a)\geq \hat{y}(t,a)$ where $\hat{y}$ is the solution of the latter system, with an equality instead. We see that $\Phi^y_{\tau}(z)\in \S_y$ by Proposition \ref{Prop:Attract} (4) so that the function
$$(0,c)\ni a\longmapsto \Phi^y_{\tau}(z)(a)$$
belongs to $L^1_+(0,c)\setminus \{0\}$ since $c>\overline{\beta_y}$. Since \eqref{Eq:R0y_small} holds, we deduce from Lemma \ref{Lemma:Semigroup} that
$$\lim_{t\to \infty}\displaystyle \int_0^c y(t,a)da\geq \lim_{t\to \infty}\displaystyle \int_0^c \hat{y}(t,a)da=\infty$$
which contradicts \eqref{Eq:Unstab_Contrad}, whence \eqref{Eq:Unstab} is proved. Therefore we obtain
\begin{equation}\label{Eq:Intersc=empty}
\{z\in \S_x\cap \S_y, \lim_{t\to \infty}\Phi_t(z)=E_0\}=\emptyset.
\end{equation}
Now, consider $z\in \S_x\cap \S_y$, and suppose by contradiction that there exists $w\in \omega(z)\cap \partial \S_x\cap \partial \S_y$. The invariance of $\omega(z)$ (due to Lemma \ref{Lemme:OmegaLimit}) then gives $\omega(w)=\omega(z)$ and so
$$d(\omega(z), E_0)\leq d(\omega(w), \Phi_t(w))+d(\Phi_t(w),E_0), \qquad \forall t\geq 0.$$
A consequence of Lemma \ref{Lemme:OmegaLimit} and Proposition \ref{Prop:Attract} (2), is that $d(\omega(z),E_0)=0$ and so $\{E_0\}\subset \omega(z)$ which contradicts \eqref{Eq:Intersc=empty}, whence $\omega(z)\subset \S_x\cup \S_y$ for any $z\in \S_x\cap \S_y$. \\
\textbf{(b)} Suppose that $R_0^x>\max\{1,R_0^y\}$. First suppose that $R_0^y\leq 1$, then using Proposition \ref{Prop:Attract} (8.a) and (7.b), we deduce that $\omega(z)\subset (\S_x \cup \S_y)\cap \partial \S_y=\S_x\cap \partial \S_y\subset \S_x$. Now suppose that $R_0^x>R_0^y>1$. We see that $r_x>r_y$, so we can consider $\ep>0$ small enough such that $r_x(1/r_y -\ep)>1$. We then define the set
$$\M_\ep:=\{(S_0,x_0,y_0)\in \S_x\cap \S_y, \|x_0\|_{L^1(\R_+)}\leq \ep\}$$
and we aim to prove that for every $z\in \M_\ep$, there exists $\overline{t}(z)$ such that
\begin{equation}\label{Eq:UnstabX}
\|\Phi_{\overline{t}}^x(z)\|_{L^1(\R_+)}>\ep.
\end{equation}
By contradiction, suppose that there exists $z:=(S_0,x_0,y_0)$ such that
\begin{equation}\label{Eq:UnstabX_Contrad}
\|\Phi_{t}^x(z)\|_{L^1(\R_+)}\leq \ep, \qquad \forall t\geq 0.
\end{equation}
Denoting $(\Phi^S_t(z),\Phi^x_t(z), \Phi^y_t(z))=(S,x,y)$ for notational simplicity, we deduce from \eqref{Eq:Model} that $S$ satisfies the following inequalities
$$S'(t)\leq \Lambda-\mu_S S(t)-S(t)\int_0^\infty \beta_y(a)y(t,a)da$$	
and
$$S'(t)\geq \Lambda-(\mu_S+\ep\|\beta_x\|_{L^\infty}) S(t)-S(t)\int_0^\infty \beta_y(a)y(t,a)da.$$
Consider the following models
\begin{equation*}
\left\{
\begin{array}{rcl}
S'(t)&=&\Lambda-\mu_S S(t)-S(t) \int_0^\infty\beta_x(a)x(t,a)da, \\
\dfrac{\partial y(t,a)}{\partial t}+\dfrac{\partial  y(t,a)}{\partial a}&=&-\mu_y(a)y(t,a), \\
y(t,0)&=&S(t) \int_0^\infty \beta_y(a)y(t,a)da
\end{array}
\right.
\end{equation*}
and
\begin{equation*}
\left\{
\begin{array}{rcl}
S'(t)&=&\Lambda-(\mu_S+\ep\|\beta_x\|_{L^\infty}) S(t)-S(t) \int_0^\infty\beta_x(a)x(t,a)da, \\
\dfrac{\partial y(t,a)}{\partial t}+\dfrac{\partial  y(t,a)}{\partial a}&=&-\mu_y(a)y(t,a), \\
y(t,0)&=&S(t) \int_0^\infty \beta_y(a)y(t,a)da
\end{array}
\right.
\end{equation*}
then using Proposition \ref{Prop:Magal}, we deduce that
$$S(t)\xrightarrow[t\to \infty]{}\dfrac{1}{r_y}$$  
where $r_y$ is defined in Section \ref{Sec:Intro}. Consequently, there exists $\tilde{t}\geq 0$ such that for every $t\geq \tilde{t}$, we have
$$S(t)\geq \dfrac{1}{r_y}-\ep.$$
By definition of $\ep$ and by continuity arguments, there exists $\overline{\beta_x}<c<\infty$ such that
\begin{equation}\label{Eq:R0x_small}
\left(\dfrac{1}{r_y}-\ep\right)\displaystyle \int_0^c \beta_x(a)\pi_x(a)da>1.
\end{equation}
From \eqref{Eq:Model}, we deduce that $x$ satisfies:
\begin{equation*}
\left\{
\begin{array}{rcl}
\dfrac{\partial x(t,a)}{\partial t}+\dfrac{\partial  x(t,a)}{\partial a}&=&-\mu_x(a) x(t,a), \vspace{0.1cm} \\
x(t,0)&\geq& \left(\dfrac{1}{r_y}-\ep\right)\displaystyle \int_0^c \beta_y(a)y(t,a)da, \\
x(\tilde{t},a)&=&\Phi^x_{\tilde{t}}(z)
\end{array}
\right.
\end{equation*}
for every $a\in[0,c]$ and every $t\geq \tilde{t}$.  We then have $x(t,a)\geq \hat{x}(t,a)$ where $\hat{x}$ is the solution of the latter system, with an equality instead of the inequality, for every $t\geq \tau$ and a.e. $a\in[0,c]$. We see that the function 
$$(0,c)\ni a\longmapsto \Phi^x_{\tilde{t}}(z)(a)$$
belongs to $L^1_+(0,c)\setminus \{0\}$ since $\Phi_{\tilde{t}}(z)\in \S_x$ from Proposition \ref{Prop:Attract} (4). Since \eqref{Eq:R0x_small} holds, we deduce by Lemma \ref{Lemma:Semigroup} that
$$\lim_{t\to \infty}\displaystyle \int_0^c x(t,a)da\geq \lim_{t\to \infty}\int_0^c \hat{x}(t,a)da=\infty$$
which contradicts \eqref{Eq:UnstabX_Contrad}, hence \eqref{Eq:UnstabX} holds. We deduce that
\begin{equation}\label{Eq:InterscX=empty}
\left\{z\in \S_x\cap \S_y: \lim_{t\to \infty}\|\Phi_t^x(z)\|_{L^1(\R_+)}=0\right\}=\emptyset.
\end{equation}
Let $z\in \S_x\cap \S_y$ and suppose that there exists $w\in \omega(z)\cap \partial \S_x$, then 
$$d(\omega(z), E_2)\leq d(\omega(w), \Phi_t(w))+d(\Phi_t(w), E_2)=0$$
by Lemma \ref{Lemme:OmegaLimit} and Proposition \ref{Prop:Attract} (5.a), whence $\{E_2\}\subset \omega(z)$ which contradicts \eqref{Eq:InterscX=empty}. Consequently we have $\omega(z)\subset \S_x$ for any $z\in \S_x\cap\S_y$. \\
\textbf{(c)} The latter argument proves that whenever $R^y_0>\max\{1, R_0^x\}$, then $\omega(z)\subset \S_y$ for any $z\in \S_x\cap \S_y$.
\end{enumerate}
\end{proof}

\begin{remark}
We can note that to prove the item 4 of the latter proposition, we may not need to assume the item 2 of Assumption \ref{Assump1}: namely the existence of $\underline{\beta_x}$ and $\underline{\beta_y}$. Indeed, we can make use of irreducible operators to prove the statement, as in \cite[Lemma 5.1]{MagClusk2013}. However we still make the assumption, since the sketch of proof would be tedious and not add much to the result.
\end{remark}

\section{Global analysis} \label{Sec:Global}

In this section, we aim to prove that the equilibria defined in Section \ref{Sec:Intro}, satisfy a global stability property. To this end, we use Lyapunov functionals.

\subsection{Lyapunov functionals} \label{Sec:Lyap}

We define
$$L_0:z\longmapsto S^*_0 g\left(\dfrac{S}{S^*_0}\right)+\int_0^\infty \Psi_x(a)x(a)da+\int_0^\infty \Psi_y(a)y(a)da$$
$$L_x:z\longmapsto S^*_1 g\left(\dfrac{S}{S^*_1}\right)+\int_0^\infty \Psi_x(a)x^*_1(a)g\left(\dfrac{x(a)}{x^*_1(a)}\right)da+\int_0^\infty \Psi_y(a) y(a)da;$$
$$L_y:z\longmapsto S^*_2 g\left(\dfrac{S}{S^*_2}\right)+\int_0^\infty \Psi_x(a) x(a)da+\int_0^\infty \Psi_y(a)y^*_2(a)g\left(\dfrac{y(a)}{y^*_2(a)}\right)da$$
for any $z=(S,x,y)\in \X$, where $\Psi_x\in L^\infty_+(0,\infty)$ and $\Psi_y\in L^\infty_+(0,\infty)$ are defined by
$$\Psi_x(a)=\dfrac{1}{r_x}\int_a^\infty \beta_x(s)e^{-\int_a^s \mu_x(\xi)d\xi}ds, \qquad \Psi_y(a)=\dfrac{1}{r_y}\int_a^\infty \beta_y(s)e^{-\int_a^s \mu_y(\xi)d\xi}ds$$
for every $a\geq 0$, and we remind that the other parameters are defined in Section \ref{Sec:Intro}. We first start with a well-posedness result:

\begin{proposition}\label{Prop:Lyapunov_Defined}\mbox{}
\begin{enumerate}
\item The function $(t,z)\mapsto L_0(\Phi_t(z))$ is well-defined on $\R\times (\R^*\times L^1_+(0,\infty)\times L^1_+(0,\infty))$;
\item for every $z\in \S_x$, the function $(t,v)\mapsto L_x(\Phi_t(v))$ is well-defined on $\R_+\times \omega(z)$ whenever $R^x_0 >\max\{1,R^y_0\}$;
\item for every $z\in \S_y$, the function $(t,v)\mapsto L_y(\Phi_t(v))$ is well-defined on $\R_+\times \omega(z)$ whenever $R^y_0 >\max\{1,R^x_0\}$;
\item let $z\in \S_x\cap \S_y$ and suppose that $R^x_0 =R^y_0 >1$. If $\omega(z)\subset \S_x$, then the function $(t,v)\mapsto L_x(\Phi_t(v))$ is well-defined on $\R_+\times \omega(z)$, if $\omega(z)\subset \S_y$, then the function $(t,v)\mapsto L_y(\Phi_t(v))$ is well-defined on $\R_+\times \omega(z)$.
\end{enumerate}
\end{proposition}

\begin{proof}\mbox{}
\begin{enumerate}
\item By Theorem \ref{Thm:Global}, we know that the semiflow $\Phi_t$ is positive, and that $\Phi_t^S>0$ for every $t>0$, so it proves the first point.
\item Suppose that $R_0^x>\max\{1,R_0^y\}$ and let $z\in \S_x$. Either $z\in \partial \S_y$, so from Proposition \ref{Prop:Attract} (6.a), we deduce that $\omega(z)\subset \S_x\cap \partial \S_y$, or $z\in \S_y$ and we deduce from Proposition \ref{Prop:Attract} (8.b) that $\omega(z)\subset \S_x$. Moreover, Proposition \ref{Prop:Attract} (3) ensures us that
$$S^*_1 g\left(\dfrac{\Phi^S_t(v)}{S^*_1}\right)$$
is well-defined for every $t\geq 0$ and every $v\in \omega(z)$. We now prove that there exists a positive constant $c(z)>0$, such that
\begin{equation}\label{Eq:Lyap_Defined}
0\leq x^*_1(a)g\left(\dfrac{\Phi_t^x(v)(a)}{x^*_1(a)}\right)\leq c(z)\Phi_t^x(v)(a)
\end{equation}
for every $a\geq 0$, $t\geq 0$ and $v\in \omega(z)$. Following \cite[Proposition 2]{Perasso2019}, we note that the definition of the function $g$ (in \eqref{Eq:Fct_g}), implies that the following inequality holds:
$$\ln(r)\leq r-1, \qquad \forall r>0.$$
Let $t\geq 0$ and $v\in \omega(z)$, then we deduce that the middle term of \eqref{Eq:Lyap_Defined} is given by
\begin{flalign*}
x^*_1(a)g\left(\dfrac{\Phi_t^x(v)(a)}{x^*_1(a)}\right)&=x^*_1(a)\left(\dfrac{\Phi_t^x(v)(a)}{x^*_1(a)}+\ln\left(\dfrac{x^*_1(a)}{\Phi_t^x(v)(a)}\right)-1\right) \\
&\leq x^*_1(a)\left(\dfrac{\Phi_t^x(v)(a)}{x^*_1(a)}+\dfrac{x^*_1(a)}{\Phi_t^x(v)(a)}-2\right)\\
&=\Phi_t^x(v)(a)\left(\dfrac{x^*_1(a)}{\Phi_t^x(v)(a)}-1\right)^2.
\end{flalign*}
Thus, to prove \eqref{Eq:Lyap_Defined}, it suffices to prove that there exists a constant $c(z)$, such that
\begin{equation}\label{Eq:Lyap_Defined2}
\left(\dfrac{x^*_1(a)}{\Phi_t^x(v)(a)}-1\right)^2\leq c(z), \qquad \forall a\geq 0.
\end{equation}
for every $t\geq 0$ and every $v\in \omega(z)$. From Proposition \ref{Prop:Attract} (4), we know that there exists $\tau\geq 0$ such that
$$\int_0^\infty \beta_x(a)\Phi_t^x(v)(a)da>0$$
for every $t\geq \tau$ and every $v\in \omega(z)$. Let $v=(v^S, v^x, v^y) \in \omega(z)$. The invariance of $\omega(z)$ under the semiflow implies that for every $\overline{t}\geq \tau$, there exists $u\in\omega(z)$ such that
$$v=\Phi_{\overline{t}}(u).$$
We deduce that
$$\int_0^\infty \beta_x(a)v^x(a)da=\int_0^\infty \beta_x(a)\Phi_{\overline{t}}^x(u)(a)da>0.$$
Since $\omega(z)$ is compact (by Lemma \ref{Lemme:OmegaLimit}), then a continuity argument ensures us with the existence of a constant $c(z)$ (independent of $v$) such that
$$\int_0^\infty \beta_x(a)v^x(a)da\geq c(z)$$
for any $v\in \omega(z)$. Since $\Phi_t(\omega(z))\subset \omega(z)$ for all $t\geq 0$, then we get
\begin{equation}\label{Eq:Lyap_Defined3}
\int_0^\infty \beta_x(a)\Phi_t^x(v)(a)da\geq c(z), \quad \forall t\geq 0, \quad \forall v\in \omega(z).
\end{equation}
Suppose that $(t,a)\in (\R_+)^2$ such that $t>a$. From \eqref{Eq:Phi_x2}, \eqref{Eq:Lyap_Defined3} and Proposition \ref{Prop:Attract} (3), we know that there exist two constants $\delta>0$ and $c(z)>0$ such that
$$\Phi_t^x(v)(a)\geq \delta c(z)\pi_x(a)$$
for every $v\in \omega(z)$. By definition of $x^*_1$ (see Section \ref{Sec:Intro}), we see that
\begin{equation}\label{Eq:Lyap_Defined4}
\dfrac{\Phi_t^x(v)(a)}{x^*_1(a)}\geq \dfrac{\delta c(z)r_x}{\mu_S(R^x_0-1)}=:k(z)>0
\end{equation}
for every $v\in \omega(z)$, and consequently
$$\left(\dfrac{x^*_1(a)}{\Phi_t^x(v)(a)}-1\right)^2\leq \dfrac{1}{k(z)^2}+1+\dfrac{2}{k(z)}<\infty$$
which proves \eqref{Eq:Lyap_Defined2} for any $v\in \omega(z)$ and every $(t,a)\in (\R_+)^2$ such that $t>a$. Now, suppose that $a\geq t$. Since $\omega(z)\subset \S_x$ is invariant under the semiflow, then using \cite[p. 26]{Smith95}, we deduce that for any $v\in \omega(z)$, there exists a full orbit $\xi\longmapsto u_v(\xi)$, for every $\xi\in \R$, passing through $v$, \textit{i.e.} satisfying:
\begin{equation*}
\left\{
\begin{array}{lll}
u_v(\xi)\in \omega(z), \quad \forall \xi\in \R, \\
u_v(0)=v, \\
\Phi_\xi(u_v(s))=u_v(\xi+s), \quad \forall (\xi,s)\in \R_+\times \R.
\end{array} 
\right.
\end{equation*}
It then suffices to consider $s\in \R$, such that $t+s>a$. Since $u_v(-s)\in \omega(z)$, we deduce from \eqref{Eq:Lyap_Defined4} that
$$\dfrac{x^*_1(a)}{\Phi_t^x(v)(a)}=\dfrac{x^*_1(a)}{\Phi_t^x(u_v(0))(a)}=\dfrac{x^*_1(a)}{\Phi_{t+s}^x(u_v(-s))(a)}\leq \dfrac{1}{k(z)}$$
which proves \eqref{Eq:Lyap_Defined2} for any $v\in \omega(z)$ and every $(t,a)\in (\R_+)^2$ such that $a\geq t$. We have then proved that \eqref{Eq:Lyap_Defined2} (and consequently \eqref{Eq:Lyap_Defined}) holds for every $(t,a,v)\in \R_+\times \R_+\times \omega(z)$. Finally, the integrability on $\R_+$ of the functions
$$a\longmapsto \Psi_x(a)\Phi_t^x(u)(a), \quad \forall (t,u)\in \R_+\times \omega(z)$$
and 
$$a\longmapsto \Psi_y(a)\Phi_t^y(u)(a), \quad \forall (t,u)\in \R_+\times \omega(z)$$
imply (by using \eqref{Eq:Lyap_Defined}) that $(t,v)\longmapsto L_x(\Phi_t(v))$ is well-defined on $\R_+\times \omega(z)$ for every $z\in \S_x$. 

\item Suppose that $R_0^y>\max\{1,R_0^x\}$ and let $z\in \S_y$. Either $z\in \partial \S_x$, so we see that $\omega(z)\subset \partial \S_x\cap \S_y$ by using Proposition \ref{Prop:Attract} (5.a), or $z\in \S_x$ and we deduce from Proposition \ref{Prop:Attract} (8.c) that $\omega(z)\subset \S_y$. Using Proposition \ref{Prop:Attract} (3), we see that 
$$S^*_2g\left(\dfrac{\Phi_t^S(v)}{S^*_2}\right)$$
is well-defined for every $t\geq 0$ and every $v\in \omega(z)$. Similar computations as for proving \eqref{Eq:Lyap_Defined} imply that there exists a positive constant $c(z)>0$ such that
$$0\leq y^*_2(a)g\left(\dfrac{\Phi_t^y(v)(a)}{y^*_2(a)}\right)\leq c(z)\Phi_t^y(v)(a)$$
for every $a\geq 0, t\geq 0$ and $v\in \omega(z)$. Finally we prove as above that the function $(t,v)\longmapsto L_y(\Phi_t(v))$ is well-defined on $\R_+\times \omega(z)$ for every $z\in \S_y$.

\item Suppose now that $R_0^x=R_0^y>1$. From Proposition \ref{Prop:Attract} (8.a), we know that $\omega(z)\subset \S_x\cup \S_y$. Consequently, either $\omega(z)\subset \S_x$ and we use the first point, to prove that the function $(t,v)\longmapsto L_x(\Phi_t(v))$ is well-defined on $\R_+\times \omega(z)$, or $\omega(z)\subset \S_y$ and we use the second point, to prove that the function $(t,v)\longmapsto L_y(\Phi_t(v))$ is well-defined on $\R_+\times \omega(z)$.
\end{enumerate}
\end{proof}
We remind the following definition:
\begin{definition}
Let $S\subset \X$. A function $L:\X\to \R$ is called a Lyapunov function if there hold that:
\begin{itemize}
\item $L$ is continuous on $\overline{S}$ (the closure of $S$ in $\X$);
\item the function $\R_+\ni t\longmapsto L(\Phi_t(z))$ is non-increasing for every $z\in S$.
\end{itemize}
\end{definition}
We now show that $L_0, L_x$ and $L_y$ are Lyapunov functionals. 

\begin{proposition}\label{Prop:Lyapunov}
The following hold:
\begin{enumerate}
\item if $\max\{R_0^x, R_0^y\}\leq 1$, then $L_0$ is a Lyapunov function on $\R^*\times L^1_+(0,\infty)\times L^1_+(0,\infty)$. Moreover, if $R_0^x\leq 1$ (resp. $R_0^y\leq 1$), then $L_0$ is a Lyapunov function on $(\R^*\times L^1_+(0,\infty)\times L^1_+(0,\infty))\cap \partial \S_y$ (resp. on $(\R^*\times L^1_+(0,\infty)\times L^1_+(0,\infty))\cap \partial \S_x$);
\item if $R^x_0 >\max\{1,R^y_0\}$ then $L_x$ is a Lyapunov function on $\omega(z)$ for every $z\in \S_x$;
\item if $R^y_0>\max\{1,R^x_0\}$ then $L_y$ is a Lyapunov function on $\omega(z)$ for every $z\in \S_y$;
\item if $R^x_0 =R^y_0 >1$, then $L_x$ is a Lyapunov function on $\omega(z)$ for every $z\in \S_x\cap \S_y$ such that $\omega(z)\subset \S_x$. Moreover, $L_y$ is a Lyapunov function on $\omega(z)$ for every $z\in \S_x\cap \S_y$ such that $\omega(z)\subset \S_y$.
\end{enumerate}
\end{proposition}

\begin{proof}\mbox{}
\begin{enumerate}
\item \textbf{(a)} Suppose that $\max\{R_0^x, R_0^y\}\leq 1$ and let $z\in\R^*\times L^1_+(0,\infty)\times L^1_+(0,\infty)$. By Proposition \ref{Prop:Lyapunov_Defined} (1), we know that $L_0(\Phi_t(z))$ is well-defined for every $t\geq 0$ and $L_0$ is continuous. We denote by $(S,x,y)$ the solution of \eqref{Eq:Model}. We now proceed in the differentiation of $L_0$ w.r.t. $t$ along \eqref{Eq:Model}. First, we see that
\begin{flalign*}
&\dfrac{\partial L_0((\Phi_t(v)))}{\partial t}\\
=& \left(1-\dfrac{S^*_0}{S(t)}\right)S'(t)+\int_0^\infty \Psi_x(a)\dfrac{\partial x(t,a)}{\partial t}da+\int_0^\infty \Psi_y(a)\dfrac{\partial y(t,a)}{\partial t}da \\
=& -\dfrac{(\Lambda-\mu_S S(t))^2}{\mu_S S(t)}-\left(1-\dfrac{S^*_0}{S(t)}\right)\left(S(t)\int_0^\infty \beta_x(a)x(t,a)da+S(t)\int_0^\infty \beta_y(a)y(t,a)da\right)\\
&-\int_0^\infty \Psi_x(a)\left(\dfrac{\partial x(t,a)}{\partial a}+\mu_x(a)x(t,a)\right)da-\int_0^\infty \Psi_y(a)\left(\dfrac{\partial y(t,a)}{\partial a}+\mu_y(a)y(t,a)\right).
\end{flalign*}
We note that
\begin{equation}\label{Eq:sigma_computations}
\Psi_x(0)=\Psi_y(0)=1, \quad \Psi_x(\infty)=\Psi_y(\infty)=0, \quad \Psi_x'=\mu_x \Psi_x-\dfrac{\beta_x}{r_x}, \quad \Psi_y'=\mu_y \Psi_y-\dfrac{\beta_y}{r_y}
\end{equation}
so after integrations by parts, we get
\begin{equation*}\label{Eq:sigma_x_comput}
\int_0^\infty \Psi_x(a)\left(\dfrac{\partial x(t,a)}{\partial a}+\mu_x(a)x(t,a)\right)da=S(t)\int_0^\infty \beta_x(a)x(t,a)da-\dfrac{1}{r_x}\int_0^\infty \beta_x(a)x(t,a)da
\end{equation*}
and 
\begin{equation}\label{Eq:sigma_y_comput}
\int_0^\infty \Psi_y(a)\left(\dfrac{\partial y(t,a)}{\partial a}+\mu_y(a)y(t,a)\right)da=S(t)\int_0^\infty \beta_y(a)y(t,a)da-\dfrac{1}{r_y}\int_0^\infty \beta_y(a)y(t,a)da.
\end{equation}
Consequently, we obtain:
\begin{equation}\label{Eq:L0_Termfinal}
\dfrac{\partial L_0((\Phi_t(v)))}{\partial t}=-\dfrac{(\Lambda-\mu_S S(t))^2}{\mu_S S(t)}+\left(\dfrac{R_0^x-1}{r_x}\right)\int_0^\infty \beta_x(a)x(t,a)da+\left(\dfrac{R_0^y-1}{r_x}\right)\int_0^\infty \beta_y(a)y(t,a)da\leq 0
\end{equation}
for any $t\geq 0$. Consequently, $L_0$ is a Lyapunov function on $\R^*\times L^1_+(0,\infty)\times L^1_+(0,\infty)$ whenever $\max\{R_0^x, R_0^y\}\leq 1$. \\
\textbf{(b)} Suppose that $R_0^x\leq 1$ and let $z\in(\R^*\times L^1_+(0,\infty)\times L^1_+(0,\infty))\cap \partial \S_y$. Then $L_0(\Phi_t(z))$ is well-defined for every $t\geq 0$, from Proposition \ref{Prop:Lyapunov_Defined} (1), and is continuous. Since $\partial \S_y$ is positively invariant by Proposition \ref{Prop:Attract} (1), it follows that
$$\int_0^\infty \beta_y(a)y(t,a)da=0$$
for any $t\geq 0$. Consequently, we deduce from \eqref{Eq:L0_Termfinal} that
\begin{equation*}
\dfrac{\partial L_0((\Phi_t(v)))}{\partial t}=-\dfrac{(\Lambda-\mu_S S(t))^2}{\mu_S S(t)}+\left(\dfrac{R_0^x-1}{r_x}\right)\int_0^\infty \beta_x(a)x(t,a)da\leq 0
\end{equation*}
for any $t\geq 0$, whence $L_0$ is a Lyapunov function on $(\R^*\times L^1_+(0,\infty)\times L^1_+(0,\infty))\cap \partial \S_y$ whenever $R_0^x\leq 1$. \\
\textbf{(c)} In the case $R_0^y\leq 1$, from \eqref{Eq:L0_Termfinal} and the fact that $\partial \S_x$ is positively invariant by Proposition \ref{Prop:Attract} (1), we deduce that
\begin{equation*}
\dfrac{\partial L_0((\Phi_t(v)))}{\partial t}=-\dfrac{(\Lambda-\mu_S S(t))^2}{\mu_S S(t)}+\left(\dfrac{R_0^y-1}{r_y}\right)\int_0^\infty \beta_y(a)y(t,a)da\leq 0
\end{equation*}
for any $t\geq 0$ and every $z\in(\R^*\times L^1_+(0,\infty)\times L^1_+(0,\infty))\cap \partial \S_x$, so $L_0$ is a Lyapunov function.
\item Suppose that $R_0^x>\max\{1,R_0^y\}$ and let $z\in \S_x$. Then $L_x$ is well-defined on $\omega(z)$ from Proposition \ref{Prop:Lyapunov_Defined} (1), and is clearly continuous. Let $v\in \omega(z)$, then
\begin{flalign*}
&\dfrac{\partial (L_x(\Phi_t(v)))}{\partial t}\\
=& \left(1-\dfrac{S^*_1}{S(t)}\right)S'(t)+\int_0^\infty \Psi_x(a)\left(1-\dfrac{x^*_1(a)}{x(t,a)}\right)\dfrac{\partial x(t,a)}{\partial t}da+\int_0^\infty \Psi_y(t,a)\dfrac{\partial y(t,a)}{\partial t}da.
\end{flalign*}
Now, we compute each term. The fact that 
$$\Lambda=\mu_S S^*_1+S^*_1\int_0^\infty \beta_x(a)x^*_1(a)da$$
leads to
\begin{flalign}
&\left(1-\dfrac{S^*_1}{S(t)}\right)S'(t) \nonumber \\
=&-\dfrac{\mu_S}{S(t)}(S(t)-S^*_1)^2+\left(1-\dfrac{S^*_1}{S(t)}\right)\left(S^*_1\int_0^\infty \beta_x(a)x^*_1(a)da-S(t)\int_0^\infty \beta_x(a)x(t,a)da\right. \nonumber \\
&-\left.S(t)\int_0^\infty \beta_y(a)y(t,a)da\right). \label{Eq:Lyap_Term1}
\end{flalign}
Now, we compute the second term:
\begin{flalign*}
\int_0^\infty \Psi_x(a)\left(1-\dfrac{x^*_1(a)}{x(t,a)}\right)\dfrac{\partial x(t,a)}{\partial t}da=-\int_0^\infty \Psi_x(a)\left(1-\dfrac{x^*_1(a)}{x(t,a)}\right)\left(\dfrac{\partial x(t,a)}{\partial a}+\mu_x(a)x(t,a)\right)da.
\end{flalign*}
We remark that
$$\left(1-\dfrac{x^*_1}{x}\right)\left(\dfrac{\partial x}{\partial a}+\mu_x x\right)=x^*_1\dfrac{d}{da}g\left(\dfrac{x}{x^*_1}\right)$$
since $(x^*_1)'=-\mu_x x^*_1$. Thus, after an integration by parts we obtain:
\begin{flalign*}
\int_0^\infty \Psi_x(a)\left(1-\dfrac{x^*_1(a)}{x(t,a)}\right)\dfrac{\partial x(t,a)}{\partial t}da=\Psi_x(0)x^*_1(0)g\left(\dfrac{x(t,0)}{x^*_1(0)}\right)+\int_0^\infty (\Psi_x x^*_1)'(a)g\left(\dfrac{x(t,a)}{x^*_1(a)}\right)da.
\end{flalign*}
since $\Psi_x(\infty)=0$. Using \eqref{Eq:sigma_computations} and the fact that $x^*_1(0)=S^*_1\int_0^\infty \beta_x(a)x^*_1(a)da$ imply that
\begin{flalign}
&\int_0^\infty \Psi_x(a)\left(1-\dfrac{x^*_1(a)}{x(t,a)}\right)\dfrac{\partial x(t,a)}{\partial t}da \nonumber \\
=& S^*_1g \left(\dfrac{S(t)\int_0^\infty \beta_x(a)x(t,a)da}{S^*_1\int_0^\infty \beta_x(a)x^*_1(a)da}\right) \int_0^\infty \beta_x(a)x^*_1(a)da -S^*_1\int_0^\infty \beta_x(a) x^*_1(a) g\left(\dfrac{x(t,a)}{x^*_1(a)}\right)da \nonumber \\
=& S(t)\int_0^\infty \beta_x(a)x(t,a)da-S^*_1\ln\left(\dfrac{S(t)\int_0^\infty \beta_x(a)x(t,a)da}{S^*_1 \int_0^\infty \beta_x(a)x^*_1(a)da}\right)\int_0^\infty \beta_x(a)x^*_1(a)da \nonumber \\
&- S^*_1\int_0^\infty \beta_x(a)x^*_1(a)da-
S^*_1\int_0^\infty \beta_x(a)\left(x(t,a)-x^*_1(t,a)\ln\left(\dfrac{x(t,a)}{x^*_1(a)}\right)-x^*_1(a)\right)da \label{Eq:Lyap_Term2}
\end{flalign}
since $S^*_1=1/r_x$. After an integration by parts, we see that the third term reads as
\begin{equation}\label{Eq:Lyap_Term3}
\int_0^\infty \Psi_y(a)\dfrac{\partial y(t,a)}{\partial t}da= S(t)\int_0^\infty \beta_y(a)y(t,a)da-\dfrac{1}{r_y}\int_0^\infty \beta_y(a)y(t,a)da. 
\end{equation}
by using \eqref{Eq:sigma_y_comput}. Now, adding \eqref{Eq:Lyap_Term1} and \eqref{Eq:Lyap_Term2}, we see that:
\begin{flalign*}
&\left(1-\dfrac{S^*_1}{S(t)}\right)S'(t)+\int_0^\infty \Psi_x(a)\left(1-\dfrac{x^*_1(a)}{x(t,a)}\right)\dfrac{\partial x(t,a)}{\partial t}da \\
=&-\dfrac{\mu_S}{S(t)}(S(t)-S^*_1)^2-\dfrac{{S^*_1}^2}{S(t)}\int_0^\infty \beta_x(a)x^*_1(a)da+S^*_1\int_0^\infty \beta_x(a)x^*_1(t,a)\left(\ln\left(\dfrac{x(t,a)}{x^*_1(a)}\right)+1\right)da \\
&-S^*_1\ln\left(\dfrac{S\int_0^\infty \beta_x(a)x(t,a)da}{S^*_1 \int_0^\infty \beta_x(a)x^*_1(a)da}\right)\int_0^\infty \beta_x(a)x^*_1(a)da-\left(1-\dfrac{S^*_1}{S(t)}\right)S(t)\int_0^\infty \beta_y(a)y(t,a)da\\
=&-S^*_1\int_0^\infty \beta_x(a)x^*_1(a)\left[\left(\dfrac{S^*_1}{S(t)}-\ln\left(\dfrac{S^*_1}{S(t)}\right)-1\right)-\ln\left( \dfrac{x(t,a)\int_0^\infty \beta_x(s)x^*_1(s)ds}{x^*_1(a)\int_0^\infty \beta_x(s)x(t,s)ds} \right)\right]da \\
&-\left(1-\dfrac{S^*_1}{S(t)}\right)S(t)\int_0^\infty \beta_y(a)y(t,a)da-\dfrac{\mu_S}{S}(S-{S^*}_1)^2.
\end{flalign*}
We remark that
\begin{flalign*}
&\int_0^\infty \beta_x(a)x^*_1(a)\ln\left( \dfrac{x(t,a)\int_0^\infty \beta_x(s)x^*_1(s)ds}{x^*_1(a)\int_0^\infty \beta_x(s)x(t,s)ds} \right)da \\
=&-\int_0^\infty \beta_x(a)x^*_1(a)g\left(\dfrac{x(t,a)\int_0^\infty \beta_x(s)x^*_1(s)ds}{x^*_1(a)\int_0^\infty \beta_x(s)x(t,s)ds}\right)
\end{flalign*}
and we deduce that
\begin{flalign}
&\left(1-\dfrac{S^*_1}{S}\right)S'(t)+\int_0^\infty \Psi_x(a)\left(1-\dfrac{x^*_1(a)}{x(t,a)}\right)\dfrac{\partial x(t,a)da}{\partial t} \nonumber \\
=&-\dfrac{\mu_S}{S}(S-S^*_1)^2-S^*_1\int_0^\infty \beta_x(a)x^*_1(a)\left[g\left(\dfrac{S^*_1}{S(t)}\right)+g\left( \dfrac{x(t,a)\int_0^\infty \beta_x(s)x^*_1(s)ds}{x^*_1(a)\int_0^\infty \beta_x(s)x(t,s)ds} \right)\right]da \nonumber\\
&-\left(1-\dfrac{S^*_1}{S(t)}\right)S(t)\int_0^\infty \beta_y(a)y(t,a)da. \label{Eq:Lyap_Term4}
\end{flalign}
Now, adding \eqref{Eq:Lyap_Term3} and \eqref{Eq:Lyap_Term4}, and recalling that $S^*_1=1/r_x$, we obtain:
\begin{flalign}
&\dfrac{\partial }{\partial t}(L_x(\Phi_t(z))) \nonumber \\
=& -\dfrac{\mu_S}{S}(S-S^*_1)^2-S^*_1\int_0^\infty \beta_x(a)x^*_1(a)\left[g\left(\dfrac{S^*_1}{S(t)}\right)+g\left( \dfrac{x(t,a)\int_0^\infty \beta_x(s)x^*_1(s)ds}{x^*_1(a)\int_0^\infty \beta_x(s)x(t,s)ds} \right)\right]da \nonumber\\
&-\left(1-\dfrac{S^*_1}{S(t)}\right)S(t)\int_0^\infty \beta_y(a)y(t,a)da+S\int_0^\infty \beta_y(a)y(t,a)da-\dfrac{1}{r_y}\int_0^\infty \beta_y(a)y(t,a)da \nonumber  \\
=&-\dfrac{\mu_S}{S}(S-S^*_1)^2-S^*_1\int_0^\infty \beta_x(a)x^*_1(a)\left[g\left(\dfrac{S^*_1}{S(t)}\right)+g\left( \dfrac{x(t,a)\int_0^\infty \beta_x(s)x^*_1(s)ds}{x^*_1(a)\int_0^\infty \beta_x(s)x(t,s)ds} \right)\right]da \nonumber  \\
&-\int_0^\infty \beta_y(a)y(t,a)\left[\dfrac{1}{r_y}-\dfrac{1}{r_x}\right]da\leq 0 \label{Eq:Lyap_TermFinalX}
\end{flalign}
for any $t\geq 0$ since $g$ is a non-negative function and the fact that
$$R_0^x>R_0^y \Longleftrightarrow \dfrac{1}{r_x}<\dfrac{1}{r_y}.$$
Consequently $L_x$ is a Lyapunov function on $\omega(z)$ for every $z\in \S_x$ when $R_0^x>\max\{1,R_0^y\}$.

\item Suppose that $R_0^y>\max\{1,R_0^x\}$ and let $z\in \S_y$. Then $L_y$ is well-defined on $\omega(z)$ from Proposition \ref{Prop:Lyapunov_Defined} (2), and is clearly continuous. Let $v\in \omega(z)$. After similar computations as above, a differentiation of $L_y$ w.r.t. $t$ along \eqref{Eq:Model} gives:
\begin{flalign}
&\dfrac{\partial }{\partial t}(L_y(\Phi_t(z))) \nonumber \\
=&-\dfrac{\mu_S}{S}(S-S^*_2)^2-S^*_2\int_0^\infty \beta_y(a)y^*_2(a)\left[g\left(\dfrac{S^*_2}{S(t)}\right)+g\left( \dfrac{x(t,a)\int_0^\infty \beta_y(s)y^*_2(s)ds}{y^*_2(a)\int_0^\infty \beta_y(s)y(t,s)ds} \right)\right]da \nonumber  \\
&-\int_0^\infty \beta_x(a)x(t,a)\left[\dfrac{1}{r_x}-\dfrac{1}{r_y}\right]da\leq 0 \label{Eq:Lyap_TermFinalY}
\end{flalign}
for any $t\geq 0$. We deduce that $L_y$ is a Lyapunov function on $\omega(z)$ for every $z\in \S_y$ whenever $R_0^y>\max\{1,R_0^x\}$. 

\item Now, suppose that $R_0^x=R_0^y>1$ and let $z\in \S_x\cap \S_y$. We know by Proposition \ref{Prop:Attract} (8.a) that $\omega(z)\subset \S_x\cup \S_y$. If $\omega(z)\subset \S_x$, then using Proposition \ref{Prop:Lyapunov_Defined} (3), we know that the function $L_x$ is well-defined on $\omega(z)$ and is continuous. Let $v\in \omega(z)$. From \eqref{Eq:Lyap_TermFinalX} we see that
\begin{flalign}
&\dfrac{\partial }{\partial t}(L_x(\Phi_t(z))) \nonumber \\
=&-\dfrac{\mu_S}{S(t)}(S(t)-S^*_1)^2-S^*_1\int_0^\infty \beta_x(a)x^*_1(a)\left[g\left(\dfrac{S^*_1}{S(t)}\right)+g\left( \dfrac{x(t,a)\int_0^\infty \beta_x(s)x^*_1(s)ds}{x^*_1(a)\int_0^\infty \beta_x(s)x(t,s)ds} \right)\right]da\leq 0 \label{Eq:Lyap_TermFinalX=Y_1}
\end{flalign}
for any $t\geq 0$ since $R_0^x=R_0^y\Longleftrightarrow r_x=r_y$. Thus $L_x$ is a Lyapunov function on $\omega(z)$ for every $z\in \S_x\cap \S_y$ such that $\omega(z)\subset \S_x$. Similarly, if $\omega(z)\subset \S_y$, we know by Proposition \ref{Prop:Lyapunov_Defined} (3) that $L_y$ is well-defined on $\omega(z)$ and is continuous. Let $v\in \omega(z)$. From \eqref{Eq:Lyap_TermFinalY} we deduce that 
\begin{flalign}
&\dfrac{\partial }{\partial t}(L_y(\Phi_t(z))) \nonumber \\
=&-\dfrac{\mu_S}{S(t)}(S(t)-S^*_2)^2-S^*_2\int_0^\infty \beta_y(a)y^*_2(a)\left[g\left(\dfrac{S^*_2}{S(t)}\right)g\left( \dfrac{y(t,a)\int_0^\infty \beta_y(s)y^*_2(s)ds}{y^*_2(a)\int_0^\infty \beta_y(s)y(t,s)ds} \right)\right]da \leq 0 \label{Eq:Lyap_TermFinalX=y^*_2}
\end{flalign}
for any $t\geq 0$. Thus $L_y$ is a Lyapunov function on $\omega(z)$ for every $z\in \S_x\cap \S_y$ such that $\omega(z)\subset \S_y$. 
\end{enumerate}
\end{proof}

\subsection{Attractiveness}

Using the Lyapunov functionals defined above, we can compute the basin of attraction of each equilibrium, by means of the Lasalle invariance principle (see \textit{e.g.} \cite[Corollary 2.3]{MiSmiThi95}).
\begin{theorem}\label{Thm:Attract}
The following hold:
\begin{enumerate}
\item if $\max\{R^x_0 , R^y_0 \}\leq 1$ then $E_0$ is globally attractive in $\X_+$;
\item if $R_0^x>\max\{1, R_0^y\}$ then $E_1$ is globally attractive in $\S_x$;
\item if $R_0^y>\max\{1, R_0^x\}$ then $E_2$ is globally attractive in $\S_y$;
\item if $R^x_0 =R^y_0 >1$, then $\{E^*_\alpha, \alpha\in[1,2]\}$ is globally attractive in $\S_x\cap \S_y$.
\end{enumerate}
\end{theorem}

\begin{proof}\mbox{}
\begin{enumerate}
\item Suppose that $\max\{R_0^x, R_0^y\}\leq 1$ and let $z\in \X_+$. By Proposition \ref{Prop:Attract} (7.a) and (7.b) we have $\omega(z)\subset \partial \S_x\cap \partial \S_y$ and $\lim_{t\to \infty}\|(\Phi_t^x,\Phi_t^y)\|_{L^1(0,\infty)\times L^1(0,\infty)}=0$. Using   \eqref{Eq:Model} we deduce that $E_0$ is globally attractive in $\X_+$.

\item Suppose that $R_0^x>\max\{1,R_0^y\}$. We use the Lasalle invariance principle to prove the global attractiveness of $E_1$ in $\S_x$. Let $z\in \S_x$. From Proposition \ref{Prop:Attract}  (6.a) and (8.b) we deduce that $\omega(z)\subset \S_x$. Consequently to Proposition \ref{Prop:Lyapunov_Defined} (2), for every $v\in \omega(z)$, the function $t\longmapsto L_x(\Phi_t(v))$ is constant. A differentiation w.r.t. $t$ implies that
$$\dfrac{d}{dt}L_x(\Phi_t(v))=0$$
for any $t\geq 0$. From \eqref{Eq:Lyap_TermFinalX}, we deduce that $\Phi_t^S(v)=S^*_1$ and $\Phi_t^y(v)=0$ for any $t\geq 0$. It follows from \eqref{Eq:Model} that $v=\{E_1\}$, whence $\omega(z)\subset \{E_1\}$ and $E_1$ is globally attractive in $\S_x$.

\item Suppose that $R_0^y>\max\{1,R_0^x\}$. Let $z\in \S_x$. We know by Proposition \ref{Prop:Attract} (5.a) and (8.c) that $\omega(z)\subset \S_y$ and by Proposition \ref{Prop:Lyapunov} (2), that $L_y$ is a Lyapunov function on $\omega(z)$. From \eqref{Eq:Model} and \eqref{Eq:Lyap_TermFinalY}, we deduce that $v=\{E_2\}$, so that $\omega(z)\subset \{E_2\}$.

\item Suppose that $R_0^x=R_0^y>1$. By Proposition \ref{Prop:Attract} (8.a) we know that $\omega(z)\subset \S_x\cup \S_y$. Suppose first that $\omega(z)\subset \S_x$. Using Proposition \ref{Prop:Lyapunov} (3), we know that $L_x$ is a Lyapunov function on $\omega(z)$. As above, the Lasalle invariance principle implies that $t\longmapsto L_x(\Phi_t(v))$ is constant for every $v\in \omega(z)$. Using \eqref{Eq:Lyap_TermFinalX=Y_1}, we obtain:
$$S(t)=S^*_1, \qquad \dfrac{x(t,a)\int_0^\infty \beta_x(s)x^*_1(s)ds}{x^*_1(a)\int_0^\infty \beta_x(s)x(t,s)ds}=1$$
for every $t\geq 0$ and every $a\geq 0$. We deduce that $v\in \{E^*_\alpha, \alpha \in[1,2]\}$, whence $\omega(z)\subset \{E^*_\alpha, \alpha\in[1,2]\}$.  Similarly, if $\omega(z)\subset \S_y$, then using the Lyapunov function $L_y$ on $\omega(z)$ and the Lasalle invariance principle, we know that $t\longmapsto L_y(\Phi_t(v))$ is constant for every $v\in \omega(z)$. Using \eqref{Eq:Lyap_TermFinalX=y^*_2}, we obtain:
$$S(t)=S^*_2=\dfrac{1}{r_y}=\dfrac{1}{r_x}=S^*_1, \qquad \dfrac{y(t,a)\int_0^\infty \beta_y(s)y^*_2(s)ds}{y^*_2(a)\int_0^\infty \beta_y(s)y(t,s)ds}=1$$
for every $t\geq 0$ and every $a\geq 0$. We deduce that $v\in \{E^*_\alpha, \alpha \in[1,2]\}$, whence $\omega(z)\subset \{E^*_\alpha, \alpha\in[1,2]\}$ and this latter set is globally attractive in $\S_x\cap \S_y$.
\end{enumerate}
\end{proof}

\begin{remark}
We can note that the first point could also be proved by using the Lyapunov functional $L_0$.
\end{remark}

\subsection{Lyapunov stability} \label{Sec:Lyap_Stab}

In this section, we handle the stability of $E_0$ in the cases where the principle of linearisation (Proposition \ref{Prop:Stability}) fails.

\begin{proposition}\label{Prop:Lyap_Stab_E0}There hold that:
\begin{enumerate}
\item if $\max\{R_0^x, R_0^y\}=1$, then $E_0$ is stable (in $\X_+$);
\item if $R_0^x\leq 1$, then $E_0$ is stable in $\partial  \S_y$;
\item if $R_0^y\leq 1$, then $E_0$ is stable in $\partial \S_x$.
\end{enumerate}
\end{proposition}
To prove this result, we need to define the following sets:
$$L_0^\eta:=\{z\in \R^*_+\times L^1_+(0,\infty)\times L^1_+(0,\infty): L_0(z)\leq \eta\}$$
and
$$B(E_0,\eta):=\{z\in \X_+:\|z-E_0\|_{\X}\leq \eta\}$$
for any $\eta>0$, and we give two lemmas (see \textit{e.g.} \cite[Proof of Theorem 1.2]{Gabriel2012} and \cite[Proposition 3.12]{PerassoRichard19b}).

\begin{lemma}\label{Lemma:Stab1}
For every $\ep>0$, there exists $\eta>0$ such that $B(E_0,\eta)\subset L^\ep_0$.
\end{lemma}

\begin{proof}
Let $\ep>0$, $\eta>0$ and $(S^\eta_0, x^\eta_0, y^\eta_0)\in B(E_0,\eta)$. We then have:
$$\left|S^\eta_0-S^*_0\right|\leq \eta, \qquad \|x^\eta_0\|_{L^1(0,\infty)}\leq \eta, \qquad \|y^\eta_0\|_{L^1(0,\infty)}\leq \eta$$
whence
$$\lim_{\eta \to 0}S^\eta_0=S^*_0, \quad \lim_{\eta \to 0}\|x^\eta_0\|_{L^1(0,\infty)}=0, \quad \lim_{\eta \to 0}\|y^\eta_0\|_{L^1(0,\infty)}=0.$$
Moreover, for $\eta>0$ small enough, we have $S^\eta_0>0$, so that $(S_0^\eta, x_0^\eta, y_0^\eta)\in \R^*_+\times L^1_+(0,\infty)\times L^1_+(0,\infty)$.
Consequently we get
$$\lim_{\eta \to 0}\dfrac{\Lambda}{\mu_S} g\left(\dfrac{S^\eta_0}{S^*_0}\right)=0,$$
$$\int_0^\infty \Psi_x(a)x^\eta_0(a)da\leq \|\Psi_x\|_{L^\infty(0,\infty)}\|x^\eta_0\|_{L^1(0,\infty)}\xrightarrow[\eta \to \infty]{}0$$
and
$$\int_0^\infty \Psi_y(a)x^\eta_0(a)da\leq \|\Psi_y\|_{L^\infty(0,\infty)}\|y^\eta_0\|_{L^1(0,\infty)}\xrightarrow[\eta \to \infty]{}0.$$
We deduce that for $\eta>0$ small enough, we have $L_0(S^\eta_0, x^\eta_0, y^\eta_0)\leq \ep$.
\end{proof}

Since $\Psi_x(0)=1$ and $\Psi_y(0)=1$, we can find $c_\psi>0$ such that
$$\underline{\Psi}:=\min_{a\in [0,c_\Psi]}\{\Psi_x(a), \Psi_y(a)\}>0.$$

\begin{lemma}\label{Lemma:Stab2}
For every $\ep>0$, there exists $\eta>0$ such that for any $(S_0, x_0, y_0)\in L^\eta_0$, we have 
\begin{equation}\label{Eq:Stab}
\left\|S_0-S^*_0\right\|\leq \ep, \quad \int_0^{c_\Psi} x_0(a)da\leq \ep, \quad \int_0^{c_\Psi} y_0(a)da\leq \ep.
\end{equation}
\end{lemma}

\begin{proof}
Let $\ep>0$, $\eta>0$ and $z:=(S^\eta_0, x^\eta_0, y^\eta_0)\in L^\eta_0$. We see that
$$S^*_0 g\left(\dfrac{S^\eta_0}{S^*_0}\right)\leq \eta, \quad \int_0^\infty \Psi_x(a)x^\ep_0(a)da\leq \eta, \quad \int_0^\infty \Psi_y(a)y^\ep_0(a)da\leq \eta.$$
We deduce that
$$\int_0^{c_\Psi}x^\eta_0(a)da\leq \dfrac{1}{\underline{\Psi}}\int_0^{c_\Psi}\Psi_x(a)x^\eta_0(a)da\leq \dfrac{\eta}{\underline{\Psi}}, \qquad \int_0^{c_\Psi}y^\eta_0(a)da\leq \dfrac{1}{\underline{\Psi}}\int_0^{c_\Psi}\Psi_y(a)y^\eta_0(a)da\leq \dfrac{\eta}{\underline{\Psi}}$$
and consequently we have
$$\lim_{\eta \to 0}\left|S^\ep_0-S^*_0\right|=0, \quad \lim_{\eta\to 0}\int_0^{c_\Psi}x^\eta_0(a)da=0, \quad \lim_{\eta\to 0}\int_0^{c_\Psi}y^\eta_0(a)da=0$$
which ends the proof. 
\end{proof}

\begin{proof}[Proof of Proposition \ref{Prop:Lyap_Stab_E0}]
Let $\delta>0$ and $\ep>0$. From Lemma \ref{Lemma:Stab2}, there exists $\eta>0$ such that for every $(S_0,x_0,y_0)\in L^\eta_0$, then \eqref{Eq:Stab} holds. Moreover, from Lemma \ref{Lemma:Stab1}, we know that there exists $\nu>0$ such that $B(E_0,\nu)\subset L^\eta_0$. We can suppose without loss of generality that $\nu\leq \ep$.

\begin{enumerate}

\item Suppose that $\max\{R_0^x, R_0^y\}=1$. Let $z:=(S_0,x_0,y_0)\in B(E_0,\nu)$. Then $z\in L^\eta_0$ and \eqref{Eq:Stab} holds. From Proposition \ref{Prop:Lyapunov_Defined} (1), we know that the function $t\longmapsto L_0(\Phi_t(z))$ is non-increasing. Consequently, the set $L^\eta_0$ is positively invariant and $\Phi_t(z)\in L^\eta_0$ for every $t\geq 0$. From Lemma \ref{Lemma:Stab2}, we obtain:
$$\left|\Phi_t^S(z)-S^*_0\right|\leq \ep, \quad \int_0^{c_\Psi}\Phi_t^x(z)(a)da\leq \ep, \quad \int_0^{c_\Psi}\Phi_t^y(z)(a)da\leq \ep, \qquad \forall t\geq 0.$$
Let $t\geq 0$. Using \eqref{Eq:Phi_x1} we get:
\begin{flalign*}
\|\Phi_t^x(z)\|_{L^1(0,\infty)}&=\int_0^t \Phi_t^x(z)(a)da+\int_t^\infty \Phi_t^x(z)(a)da \\
&\leq \sum_{n=0}^{N_t}\int_{n c_{\Psi}}^{(n+1)c_\Psi} \Phi_t^x(z)(a)da +\|x_0\|_{L^1(0,\infty)}e^{-\mu_0 t} \\
&\leq \sum_{n=0}^{N_t}\int_0^c \Phi_{t-n c_{\Psi}}^x(z)(a)e^{-\mu_0 n c_{\Psi}}da+\nu e^{-\mu_0 t} \\
&\leq \ep \sum_{n=0}^{N_t} (e^{-\mu_0 c_{\Psi}})^n+\ep e^{-\mu_0 t}\leq \dfrac{\ep}{1-e^{-\mu_0 c_\Psi}}+\ep
\end{flalign*}
where $N_t=[\frac{t}{c_\Psi}]$ is the integer part of $\frac{t}{c_\Psi}$. Likewise, we get 
$$\|\Phi_t^y(z)\|_{L^1(0,\infty)}\leq \dfrac{\ep}{1-e^{-\mu_0 c_\Psi}}+\ep.$$
It follows that
$$\|\Phi_t(z)-E_0\|_{\X}\leq 3\ep+\dfrac{2\ep}{1-e^{-\mu_0 c_\Psi}}$$
for every $t\geq 0$. Finally, considering $\ep>0$ such that
$$3\ep +\dfrac{2\ep}{1-e^{-\mu_0 c_\Psi}}\leq \delta$$
proves the stability of $E_0$.

\item Suppose that $R_0^x>1$. Let $z:=(S_0,x_0,y_0)\in B(E_0,\nu)\cap \partial \S_y$. The former arguments and the fact that the function $t\longmapsto L_0(\Phi_t(z))$ is non-increasing imply that $E_0$ is stable in $\partial \S_y$ whenever $R_0^x>1$.

\item It follows from the last point and interchanging the index $x$ and $y$.
\end{enumerate}
\end{proof}

While the stability of $E_0$ in the critical cases are handled in the latter proposition, the question of the stability of the set $\{E^*_\alpha, \alpha\in[1,2]\}$ when $R_0^x=R_0^y>1$ is open. The use of Lyapunov functional in the latter proof will raise some problems due to the fact that $L_x$ and $L_y$ are not defined in $\X_+$.

%\begin{proposition}\label{Prop:Lyap_Stab_E*}
%Suppose that $R_0^x=R_0^y>1$. Then the set of equilibria $\{E^*_\alpha ,\alpha\in[1,2]\}$ is stable.
%\end{proposition}
%
%\begin{proof}
%???
%\end{proof}

\subsection{Global asymptotic stability}

We are ready to give the main result of the paper:

\begin{theorem}\label{Thm:GAS}
The following hold:
\begin{enumerate}
\item $E_0$ is G.A.S. in $\partial \S_x\cap \partial \S_y$. Moreover, it is also G.A.S. in
\begin{enumerate}
\item $\X_+$ if $\max\{R^x_0 , R^y_0 \}\leq 1$;
\item $\partial \S_y$ if $R_0^x\leq 1$;
\item $\partial \S_x$ if $R_0^y\leq 1$.
\end{enumerate}
\item $E_1$ is G.A.S. in:
\begin{enumerate}
\item $\S_x$ if $R_0^x>\max\{1, R_0^y\}$;
\item $\S_x\cap \partial \S_y$ if $R^x_0 >1$;
\end{enumerate} 
\item $E_2$ is G.A.S. in:
\begin{enumerate}
\item $\S_y$ if $R_0^y>\max\{1, R_0^x\}$;
\item $\partial \S_x\cap \S_y$ if $R^y_0 >1$;
\end{enumerate} 
\item if $R^x_0 =R^y_0 >1$, then $\{E^*_\alpha, \alpha\in[1,2]\}$ is globally attractive in $\S_x\cap \S_y$.
\end{enumerate}
\end{theorem}

\begin{proof}\mbox{}
\begin{enumerate}
\item The fact that $E_0$ is G.A.S. in $\partial \S_x\cap \partial \S_y$ follows from Proposition \ref{Prop:Attract} (2).  \\
\textbf{(a)} Suppose that $\max\{R_0^x, R_0^y\}\leq 1$. From Proposition \ref{Thm:Attract} (1), we know that $E_0$ is globally attractive in $\X_+$. Using Proposition \ref{Prop:Stability} (1) and Proposition \ref{Prop:Lyap_Stab_E0}, we deduce that $E_0$ is Lyapunov stable, whence the global asymptotic stability in $\X_+$. \\
\textbf{(b)} Suppose that $R_0^x\leq 1$. It follows from Proposition \ref{Prop:Attract} (6.b) that $E_0$ is globally attractive in $\partial \S_y$, and from Proposition \ref{Prop:Lyap_Stab_E0} that $E_0$ is stable in $\partial \S_y$. \\
\textbf{(c)} When $R_0^y\leq 1$, the result follows from Proposition \ref{Prop:Attract} (5.b) and Proposition \ref{Prop:Lyap_Stab_E0}.

\item \textbf{(a)} Suppose that $R_0^x>\max\{1,R_0^y\}$. The stability of $E_1$ follows from Proposition \ref{Prop:Stability} (2), while the global attractiveness in $\S_x$ comes from Theorem \ref{Thm:Attract} (2). \\
\textbf{(b)} Suppose that $R_0^x>1$. We know by Proposition \ref{Prop:Attract} (6.a) that $E_1$ is globally attractive in $\S_x\cap \partial \S_y$. Moreover, let $z:=(S_0,x_0,y_0)\in \S_x\cap \partial \S_y$ and denote by $(S,x,y)\in \Co(\R_+,\X_+)$ the solution of \eqref{Eq:Model}. Since $\partial \S_y$ is positively invariant by Proposition \ref{Prop:Attract} (1), it follows that
$$\int_0^\infty \beta_y(a)y(t,a)da=0$$
for any $t\geq 0$, so that $(S,x)$ satisfies \eqref{Eq:Model_1pop}. Let $\ep>0$. Since $(S^*_1,x^*_1)$ is Lyapunov table in $\{(S_0,x_0)\in \R_+\times L^1_+(\R_+): \int_0^{\overline{\beta_x}}x_0(s)s>0\}$ by Proposition \ref{Prop:Magal}, then we can find $\eta>0$ such that 
$$\|(S_0,x_0)-(S^*_1,x^*_1)\|_{\R\times L^1(0,\infty)}\leq \eta \Rightarrow \|(\Phi_t^S(z), \Phi_t^x(z))-(S^*_1,x^*_1)\|_{\R\times L^1(0,\infty)}\leq \dfrac{\ep}{2}.$$
We also know that $\|\Phi_t^y(z)\|_{L^1(0,\infty)}\leq e^{-\mu_0 t}\|x_0\|_{L^1(0,\infty)}$ for any $t\geq 0$ by using Proposition \ref{Prop:Attract} (6). Thus we consider $\tilde{\eta}:=\min\{\eta,\ep/2\}$ and we let $z\in \S_x\cap \partial \S_y$ such that $\|z-E_1\|_{\X}\leq \eta$. We then have
$$\|\Phi_t(z)-E_1\|_{\X}=\|(\Phi_t^S(z),\Phi_t^x(z))-(S^*_1,x^*_1)\|_{\R\times L^1(0,\infty)}+\|\Phi_t^y(z)\|_{L^1(0,\infty)}\leq \ep$$
which proves the Lyapunov stability of $E_1$ in $\partial \S_y$ and consequently the global stability.

\item \textbf{(a)} Suppose that $R_0^y>\max\{1,R_0^x\}$. From Proposition \ref{Prop:Stability} (3) and Theorem \ref{Thm:Attract} (3) we deduce that $E_2$ is G.A.S. in $\S_x$. \\
\textbf{(b)} Similarly, when $R_0^y>1$ the global stability is deduced from Proposition \ref{Prop:Attract} (5.a).

\item Suppose that $R_0^x=R_0^y>1$, then the result derives from Theorem \ref{Thm:Attract} (4).
\end{enumerate}
\end{proof}
We can note that the global stability of endemic equilibria implies the persistent of the corresponding disease. 

\section{Numerical simulations and final remarks} \label{Sec:Simu}

We start this section by some illustrations of the main results. We plot the total quantity of individuals, \textit{i.e.} the $L^1$-norm for $x$ and $y$, in function of time. We also consider two different initial conditions (in line and dotted line) in $\S_x\cap\S_y$. In Figures \ref{Fig:Case1} and \ref{Fig:Case2}, the competitive exclusion principle applies: the disease with the biggest $R_0$ value persists while the other one go extinct. In Figure \ref{Fig:Case3}, the two solutions (corresponding to both initial conditions), converge to two different equilibria belonging to the set $\{E^*_\alpha ,\alpha \in[1,2]\}$. We can note that the results obtained in the paper could be extended to the general case $(N\geq 3)$:
\begin{equation*}
\left\{
\begin{array}{rcl}
\dfrac{dS}{dt}(t)&=&\Lambda-\mu_S S(t)-S(t) \displaystyle \sum_{n=1}^N \int_0^\infty  \beta_{x_n}(a)x_n(t,a)da, \\
\dfrac{\partial x_n}{\partial t}(t,a)+\dfrac{\partial x_n}{\partial a}(t,a)&=&-\mu_{x_n}(a)x_n(t,a), \\
x_n(t,0)&=&S(t) \int_0^\infty \beta_{x_n}(a)x_n(t,a)da, \vspace{0.1cm} \\
(S(0), x_1(0,\cdot), \cdots, x_{N}(0,\cdot))&=&(S^0, x^0_1, \cdots, x^0_N)\in \R_+\times (L^1_+(0,\infty))^N
\end{array}
\right.
\end{equation*}
for every $n\in\llbracket 1,N\rrbracket$. As we noticed with \eqref{Eq:Model}, considering an initial condition in $\partial S_{x_n}$ for some $n\in\llbracket 1,N \rrbracket$ amounts to study the $N-1$ dimensional case. Therefore, even if the number of cases increase exponentially, only the set $\S_{x_1}\times \cdots \times \S_{x_N}$ is important for the initial conditions. In that situation, the competition exclusive principle applies whenever there exists $i\in\llbracket 1,N\rrbracket$ such that $R_0^{x_i}>R_0^{x,j}$ for every $j\in\llbracket 1,N\rrbracket \setminus \{i\}$, that is: the disease $x_i$ persists while all the other go extinct. When the maximum if not unique, we can prove the existence of an infinite number of equilibria, that constitute a global attractive set, whose stability is an open problem.

\begin{figure}[hbtp]
   \begin{minipage}[c]{.46 \linewidth}
      \begin{center}
\includegraphics[scale=0.19]{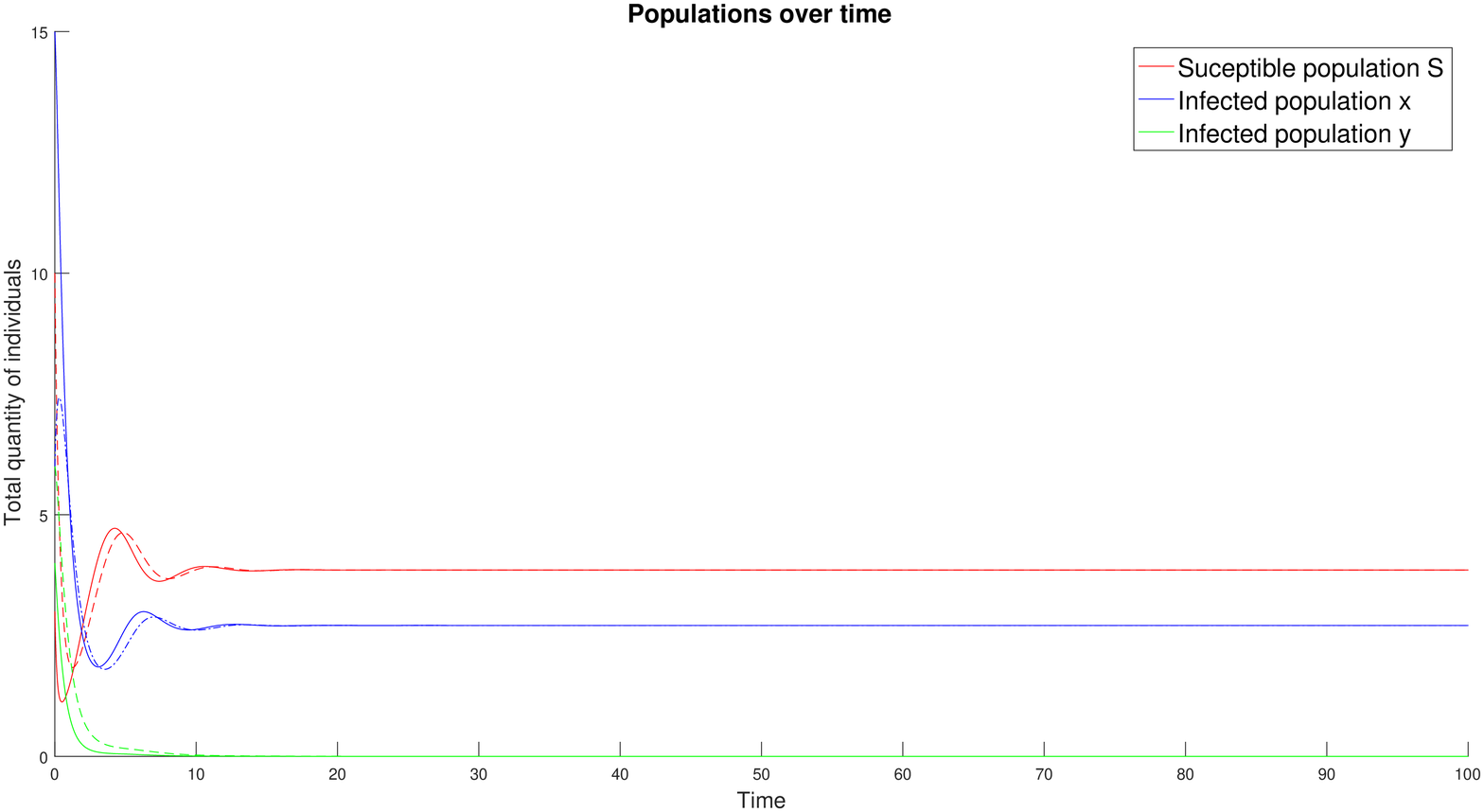}
\caption{Case $R_0^x>\max\{R_0^y,1\}$}
      \label{Fig:Case1}
      \end{center}
   \end{minipage} \hfill
   \begin{minipage}[c]{.46\linewidth}
      \begin{center}      
\includegraphics[scale=0.19]{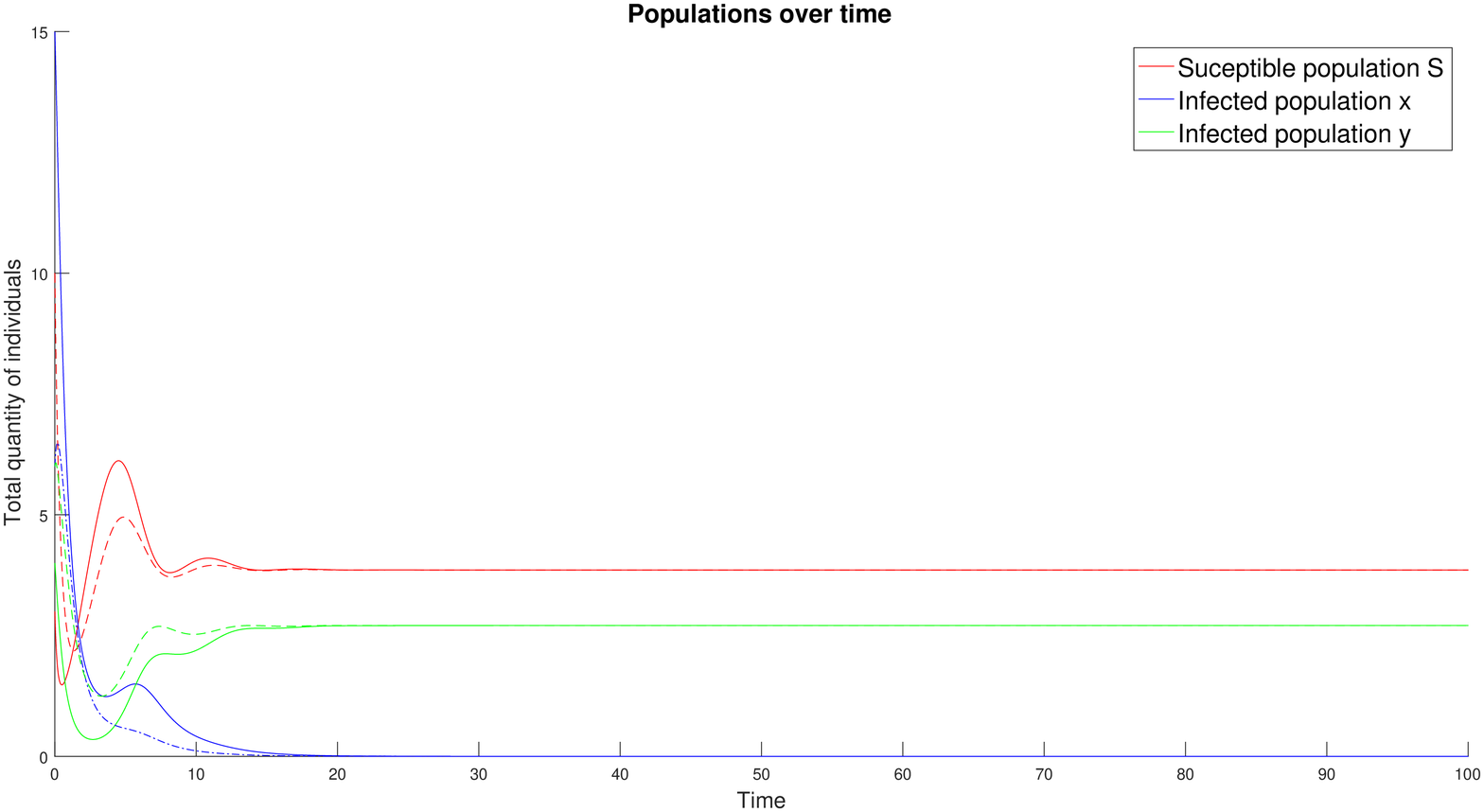}
\caption{Case $R_0^y>\max\{R_0^x,1\}$}
      \label{Fig:Case2}
      \end{center}
   \end{minipage}
\end{figure}

\begin{figure}
\begin{center}
\includegraphics[scale=0.2]{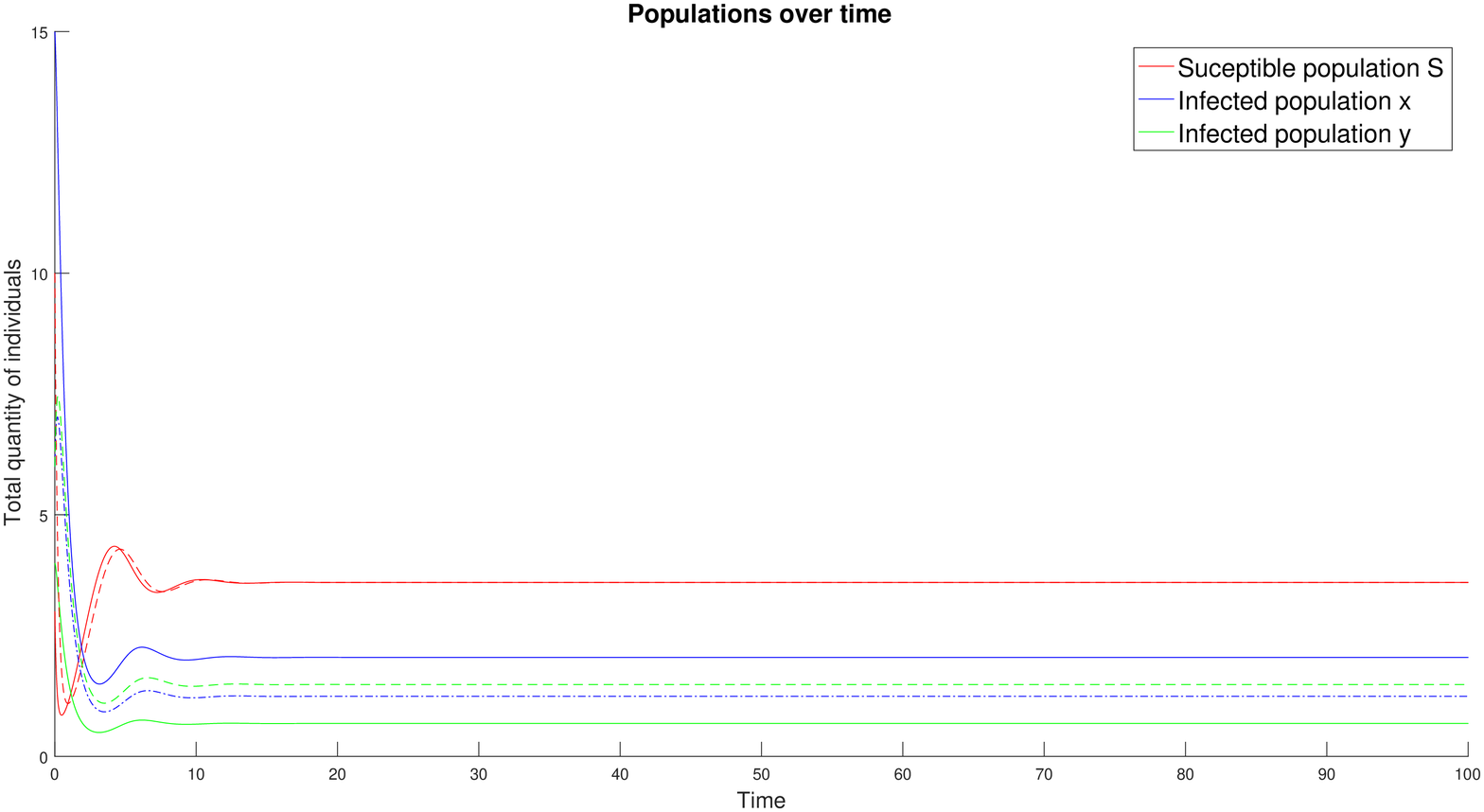}
\caption{Case $R_0^x=R_0^y>1$}
\label{Fig:Case3}
\end{center}
\end{figure}

%%-----------------------------
%%      your bibliography
%%-----------------------------

\end{document}